\documentclass{calabi}

\usepackage{euscript}
\usepackage{parskip}
\usepackage{leftindex}

\usepackage{tikz-cd}

\newcommand{\bQ}{\mathbb{Q}}
\newcommand{\bN}{\mathbb{N}}
\newcommand{\bP}{\mathbb{P}}
\newcommand{\bC}{\mathbb{C}}
\newcommand{\bR}{\mathbb{R}}
\newcommand{\bZ}{\mathbb{Z}}

\newcommand{\bK}{\mathbb{K}}

\newcommand{\Hess}{\mathrm{Hess}}

\newcommand{\val}{\mathrm{val}}

\begin{document}

\title{Orbifold Floer spectral invariants, symmetric product links \\ and  Weyl laws}
\date{\today}
\author{Cheuk Yu Mak, Sobhan Seyfaddini, and Ivan Smith}
\maketitle

\begin{abstract} 
\noindent 
We explain a strategy, based on spectral invariants on symmetric product orbifolds,  for proving the smooth closing lemma for Hamiltonian diffeomorphisms of a symplectic manifold when the orbifold quantum cohomologies of its symmetric products possess suitable idempotents. We relate the existence of such idempotents to the manifold containing a sequence of Lagrangian links, whose number of components tends to infinity, satisfying a number of properties. Orbifold Floer cohomology for global quotient orbifolds is used axiomatically, and is constructed in the companion paper \cite{MSS}.  We illustrate this strategy by giving a new proof of the smooth closing lemma for area-preserving diffeomorphisms of the 2-sphere. The construction of suitable Lagrangian links in higher dimensions remains an intriguing open problem.
\end{abstract}

\tableofcontents

\section{Introduction}

\sloppy 

\subsection{Overview}

Let $(M,\omega_M)$ be a closed symplectic manifold such that $[\omega_M] \in H^2(M,\mathbb{Q})$.
This paper explains a number of applications of orbifold Floer cohomology, applied to the symmetric product orbifolds $Y:=\Sym^k(M)$, $k=1,2,\ldots$, to Hamiltonian dynamics.  The companion to this paper gives a construction of bulk-deformed orbifold Hamiltonian Floer cohomology for global quotients which readily yields our required properties. The paper \cite{Cho-Poddar} constructed orbifold Lagrangian Floer theory, and we summarise their results, and axiomatise the mild extensions thereof that our applications require (primarily, Assumption \ref{a:OCCO}), in the final section of the paper. See Section \ref{sec:structure} for a brief explanation.

Thus, this paper 
has two goals:

\begin{enumerate} 
\item We give a `template method' for proving a Weyl law for the spectral invariants on $M$ associated to a sequence of idempotents in quantum cohomologies of symmetric product orbifolds $\Sym^k(M)$.  
\item We  give a new proof of the `smooth closing lemma' for Hamiltonian symplectomorphisms of $S^2$, using idempotents associated to Lagrangian links comprising $k$ parallel circles, with $k\to\infty$.
\end{enumerate}

At present, we do not know how to construct Lagrangian links satisfying our requirements on higher-dimensional symplectic manifolds. We hope that the `axiomatic' presentation of our results will help focus attention on the importance of finding Floer-theoretically non-trivial links with arbitrarily small components, which can therefore potentially probe the small-scale geometry of $M$.   Lagrangian links satisfying our particular requirements would necessarily be {\it space filling}, in the sense that the union of their components becomes dense. 

 A fundamental question asks for which open sets $U\subset M$ and positive values $e>0$ can one create a periodic orbit meeting $U$ by perturbing a Hamiltonian by an amount $<e$ (measured in $C^{\infty}$-norm). If both $U,e$ can be taken arbitrarily small, one obtains the smooth closing lemma.  Classical Floer theory lets one take $U$ to be a neighborhood of a Floer-theoretically non-trivial Lagrangian; in particular such $U$ are non-displaceable subsets. Using Lagrangian links in toric surfaces \cite{MS19, Pol-Shel21}, we give new examples where $U$ is displaceable and $e$ is essentially the displacement energy of $U$ (see Theorem \ref{t:links and orbits}). In the case of $S^2$, Lagrangian links comprising  parallel circles accumulating at the poles of the sphere, allow one to take $U, e$ to be arbitrarily small,  hence yielding another proof of the smooth closing lemma on $S^2$; see Remark \ref{rem:links and orbits}.

Herman’s well-known counterexample \cite{Herman2, Herman1} demonstrates that the $C^\infty$ closing lemma and Weyl laws fail to hold on general symplectic manifolds of dimension greater than two. However, whether the closing lemma holds for Hamiltonian diffeomorphisms of the ball remains an important open question. The prevailing perspective \cite{Fish-Hofer20} suggests that in Hamiltonian systems, the closing lemma is tied to the richness of the symplectic topology of the ambient space, as witnessed for instance, in quantum cohomology. In this work, we establish a connection between the closing lemma and the structure of orbifold quantum cohomology on symmetric products.

\subsection{Spectral invariants for global quotient orbifolds}

In \cite{MSS} we define an orbifold version of bulk-deformed Hamiltonian Floer cohomology for a global quotient orbifold $Y = [X / \Gamma]$ where $\Gamma$ is a finite group acting by symplectomorphisms on the smooth symplectic manifold $X$ so that $\omega_X$ descends to an orbifold symplectic form $\omega_Y$ on $Y$.  We used this to construct spectral invariants on $C^\infty([0,1] \times Y)$.
We will apply it to the setting $X=M^k$ and $Y=\Sym^k(M)$.

 To state the main output of \cite{MSS}, we first formulate some background.


An orbifold $Y$ has an inertia stack $IY$, which splits into a union of components called `twisted sectors'.  For a global quotient $Y = [X/\Gamma]$, the twisted sectors are indexed by the set of conjugacy classes $\|\Gamma\|$ of $\Gamma$, with the conjugacy class of $g$ giving the sector $X^g / C(g)$ with $X^g \subset X$ the fixed point locus of $g$ and $C(g) \subset \Gamma$ the centraliser of $g$.  
Let $\|\Gamma\|^{\circ}:=\|\Gamma\|\setminus\{id\}$ be the set of non-trivial conjugacy classes.
For any $g \in \Gamma$, let $(g)$ be the conjugacy class of $g$.

We assume that $\omega_X$ (and hence $\omega_Y$) represent rational cohomology classes. Let $\mathbb{K}$ be a characteristic $0$ field, $\Pi \subset \bR$ be a discrete subgroup which contains  $\frac{1}{|\Gamma|}\omega_X(H_2(X;\mathbb{Z})) \subset \mathbb{R}$,  and 
\[\Lambda_\Pi:=\left\{\sum_{i=1}^{\infty} a_iT^{b_i} \, \big| \, a_i \in \mathbb{K},  b_i \in \Pi, \lim_i b_i=\infty \right\}.\]
Consider the valuation
\[
\val:\Lambda_{\Pi} \to \mathbb{R} \cup \{\infty\}, \quad \val\left(\sum a_iT^{b_i} \right)=\min\{b_i\}. 
\]
Let $\Lambda_{>0}:=\val^{-1}(\mathbb{R}_{>0} \cup \{\infty\})$ 
and $\Lambda_{\ge 0}:=\val^{-1}(\mathbb{R}_{ \ge 0} \cup \{\infty\})$.

We will consider a bulk deformation class which is a linear combination of the fundamental class of the non-trivial sectors so it is of the form
\begin{align}\label{eq:bulkb}
\frak{b}=\sum_{(g) \in \|\Gamma\|^{\circ}} \frak{v}_{(g)} PD[X^g/C(g)] \in H_{orb}(Y;\Lambda_{>0}),  \quad \frak{v}_{(g)} \in \Lambda_{>0}
\end{align}
where $PD$ stands for the Poincar\'e dual.

To define the action spectrum in which our spectral invariants take values, we briefly recall the construction of the chain groups underlying orbifold Hamiltonian Floer theory for a smooth $H_Y: Y \times [0,1] \to \bR$. A generator of $CF^*(H_Y;\frak{b})$ is an oriented representable $1$-periodic Hamiltonian orbits of $H_Y$ in $Y$ together with a formal capping. To spell this out:
\begin{enumerate}
    \item The orbit consists of a
commutative diagram  of the form 
\begin{align}\label{eq:orbit}
\xymatrix{
\hat{S} \ar[r]^{\hat{x}} \ar[d]^{\pi_x} & X \ar[d] \\ S^1 \ar[r]^{x} & Y
}    
\end{align}
such that $\pi_x$ is the quotient map of a free $\Gamma$-space $\hat{S}$, $x$ is a Hamiltonian orbit of $H_Y$ and $\hat{x}$ is $\Gamma$-equivariant.
We require that $x$ is orientable in the sense that when the isotropy group of $x$ is non-trivial, the action of the isotropy group on the orientation line of $x$ preserves the orientation.
\item The capping disc is a representable orbifold disc  $\bar{c}_x: D^2 \to Y$ such that $\bar{c}_x|_{\partial D^2}=x(-t)$.\footnote{The minus sign is due to the convention in \cite{MSS}.} We denote by $k_{(g)}(\bar{c}_x)$ the number of orbifold points of $\bar{c}_x$ whose evaluation map go to $X^g/C(g)$.
\end{enumerate}

The action of a capped orbit is defined to be
\begin{align}\label{eqn:action of capping}
\cA(\bar{c}_x):=\int_x H_Y dt-\int_{D^2} \bar{c}_x^*\omega_Y - k_{(g)}(\bar{c}_x)\cdot \val(\frak{v}_{(g)})    
\end{align}

If $\bar{c}_x$ and $\bar{c}'_x$ are two different cappings of $x$, then we can reverse the orientation of $\bar{c}_x$ and glue to $\bar{c}'_x$ to get on orbifold sphere $A$.
Then $\cA(\bar{c}_x)-\cA(\bar{c}'_x) \in \omega_Y(A)+\sum_{(g)} \val(\frak{v}_{(g)}) \mathbb{Z} \subset \Pi$.

We will equip our caps with a further piece of data, a real number $U \in \Pi$.
 We use the notation $c_x$ to denote $(x,\bar{c}_x,U)$ or $(\bar{c}_x,U)$.
We call $c_x$ a \emph{formal capping} of $x$ and define the action of $c_x=(\bar{c}_x,U)$ to be
\[
\cA(c_x):=\cA(\bar{c}_x)-U
\]
Two formal capped marked orbits $c_x=(\bar{c}_x, U_x)$ and $c_y=(\bar{c}_y,  U_y)$ are  equivalent if $x=y$ and
\begin{align}
\cA(c_x)= \cA(c_y)
\end{align} 

Finally, we have:

\begin{definition}\label{d:action}
The \emph{action spectrum} $\Spec(Y,H_Y)$ is the set of values of the action over all formal capped orbits. 
\end{definition} 

With this in hand, we can state:

\begin{theo}[see \cite{MSS}]\label{t:spectral_intro}
Let $Y=[X/\Gamma]$ be a global quotient orbifold. 
Let $\frak{b}$ be of the form \eqref{eq:bulkb}, and $H_Y\in C^{\infty}(S^1 \times Y)$ be a non-degenerate Hamiltonian function.

There is a $\frak{b}$-deformed orbifold Hamiltonian Floer cohomology group $HF^*(H_Y;\frak{b})$ with an action filtration (by $HF(H_Y;\frak{b})^{<a}$ for $a \in \mathbb{R}$), well-defined up to canonical isomorphism (i.e. independent of the choice of almost complex structure and independent of the presentation of $Y$ as a global quotient), which is additively isomorphic to the $\mathfrak{b}$-deformed orbifold quantum cohomology $QH_{orb}(Y,\mathfrak{b})$ (which in turn is additively isomorphic to $QH_{orb}(Y)$ with $\frak{b}=0$). 

There is an action decreasing pair of pants product
 \[
HF(H_Y;\frak{b}) \otimes HF(H'_Y;\frak{b}) \longrightarrow HF(H_Y\# H'_Y;\frak{b})
\]
    which depends on the bulk $\frak{b}$.
    
For any class $x \in QH_{orb}(Y,\mathfrak{b})$, there is an associated spectral invariant $c^{\frak{b}}(x,-) = c(x,-):C^{\infty}(S^1 \times Y) \to \mathbb{R}$ with the following properties.

\begin{enumerate}
   \item (Symplectic invariance) $c^{\frak{b}}(x,H_Y\circ \psi) = c^{\frak{b}}(x,H_Y)$ for any $\psi \in \Symp(Y, \omega_Y)$;
    \item (Spectrality) for any $H_Y$, $c^{\frak{b}}(x,H_Y)$ lies in the action spectrum $\Spec(Y,H_Y)$;
\item (Hofer Lipschitz) for any $H_Y, H'_Y$
$$  \int_{0}^1 \min  ((H_Y)_t - (H'_Y)_t) dt \leq c^{\frak{b}}(x,H_Y) - c^{\frak{b}}(x,H'_Y) \leq  \int_{0}^1 \max\,  ((H_Y)_t - (H'_Y)_t) dt;$$
\item (Monotonicity) if $(H_Y)_t \leq (H'_Y)_t$ then $c^{\frak{b}}(x,H_Y) \leq c^{\frak{b}}(x,H'_Y)$;
\item (Homotopy invariance) if $H_Y, H'_Y$ are mean-normalized and determine the same point of the universal cover $\widetilde{\Ham}(Y,\omega_Y)$, then $c^{\frak{b}}(x,H_Y) = c^{\frak{b}}(x,H'_Y)$;
\item (Shift) $c^{\frak{b}}(x,H_Y + s(t)) = c^{\frak{b}}(x,H_Y) +  \int_0^1 s(t)\, dt$;
\item (Subadditivity) for any $x,x' \in QH_{orb}(Y,\mathfrak{b})$ and $H_Y,H'_Y$, $c^{\frak{b}}(x \ast_{\frak{b}} x',H_Y \# H'_Y) \leq c^{\frak{b}}(x,H_Y) + c^{\frak{b}}(x',H'_Y)$, where $\ast_{\frak{b}}$ is the bulk-deformed orbifold quantum product in $QH_{orb}(Y,\mathfrak{b})$.
\end{enumerate}
\end{theo}

 In item (7) above, $H_Y \# H'_Y$ is referred to as the {\it composition} of the Hamiltonians $H, H'$ and is defined via the formula
\begin{align}\label{eq:comp}
H_Y \# H'_Y(t,x):= H_Y(t,x) + H'_Y(t, (\varphi^t_{H_Y})^{-1}(x)).
\end{align}
The Hamiltonian flow of $ H_Y \# H'_Y$ is the composition of the flows of $H_Y$ and $H'_Y$, i.e.\ $\varphi^t_{H_Y\# H'_Y} = \varphi^t_{H_Y} \circ \varphi^t_{H'_Y}$.

\begin{remark}
    Orbifold quantum cohomology is typically $\bQ$-graded with an appropriate Novikov field. Both it and $HF^*(H;\frak{b})$ admit a parity decomposition, compatible with the orbifold Floer product; we do not consider any grading beyond this parity decomposition.
    \end{remark}
\begin{remark}
    Even though  $\frak{b}$ in Theorem \ref{t:spectral_intro} can be a linear combination of fundamental classes of twisted sectors, below we take the simplest case in which the bulk is a multiple of a single sector, since that suffices for the applications in this paper.
\end{remark}

 \subsection{Symmetric products and Weyl laws}
Let $(M, \omega)$ be a closed symplectic manifold of dimension $2n$.  We say a sequence  $c_k: C^{\infty}([0,1]\times M) \rightarrow \bR$ of spectral invariants \emph{satisfies the Weyl law} if for every Hamiltonian $H \in C^{\infty}([0,1]\times M)$, the sequence $\frac 1k c_k(H)$  converges to the average of $H$, i.e.\
\begin{align}\label{eqn:Weyl-defn}
    \lim_k \frac 1k  c_k(H) = \frac{1}{{\mathrm{Vol}(M)}}\int_0^1 \int_M H_t \;  \omega^n \, dt,
\end{align}
where $\mathrm{Vol}(M):= \int_M \omega^n $.

In dimension $2$, every closed symplectic manifold $M$ admits spectral invariants $c_k$ satisfying  (a variation of) the above Weyl law; see \cite{CGHS20, EH21, CPZ21, CGHMSS,Buhovsky23}.  Applications of these two dimensional Weyl laws include the smooth closing lemma for area-preserving maps and non-simplicity of the group of area-preserving homeomorphisms of the disc.

A primary motivation for our work is the search for spectral invariants satisfying Weyl laws in higher dimensions, which would yield generalizations of the aforementioned applications to these settings; see Section \ref{sec:Weyl law applications} for a discussion of these applications.

The Floer cohomology of Hamiltonians on $M$ itself does not yield a rich enough source of spectral invariants for this to be very helpful.  Inspired by \cite{CGHMSS, CGHMSS22, Pol-Shel21} we use spectral invariants associated to symmetric product orbifolds. 

\begin{remark}
    Weyl laws have previously been proved using spectral invariats associated to periodic Floer homology (PFH) \cite{CGHS20, EH21, CPZ21}. In this theory, one counts possibly disconnected and higher genus holomorphic curves in the symplectic mapping torus of a Hamiltonian surface diffeomorphism. One feature of passing to the symmetric product orbifold is that Floer cylinders in the symmetric product can be interpreted, via the `tautological correspondence', as possibly higher genus and possibly disconnected curves mapping to the surface itself. On the other hand, the symmetric product theory depends subtly and importantly on the value of the bulk  deformation parameter, an analogue of which seems harder to identify in the PFH setting.  

    Another theory which counts higher genus curves, and is at least philosophically related to symmetric products, is the theory of higher dimensional Heegaard Floer homology and Heegaard Floer symplectic homology \cite{CHT, KY};  but they also do not explicitly introduce orbifolds.
\end{remark}

To obtain a sequence of spectral invariants on a closed symplectic manifold $(M, \omega_M)$ (with $[\omega_M] \in H^2(M,\mathbb{Q})$), we apply Theorem \ref{t:spectral_intro} in the  case where $X = M^k$ and $Y = \Sym^k(M)$.  In our later applications, it will also be important to allow the discrete group $\Pi = \Pi_k$ containing the valuations of the coefficients of the bulk class $\frak{b}$ to also depend on $k$, so we fix such a collection of groups $\Pi_k$.

\begin{remark}\label{rmk:choice of Pi}
    In the application, we will have a sequence of Lagrangian links $L_k \subset M$ of $k$ components, and we will take $\Pi_k$ to be a discrete group containing both $\langle \omega_M, H_2(M;\bZ)\rangle$ and $\langle \omega_M, H_2(M,L_k;\bZ)\rangle$. (The union of these discrete groups over all $k$ will typically no longer be discrete.)
\end{remark}

Via the canonical embedding 
\begin{align}\label{eq:can}
C^{\infty}([0,1]\times M) &\to C^{\infty}([0,1] \times \Sym^k(M))\\
H & \mapsto \Sym^k(H)_t(x_1,\dots,x_k):=\sum_{i=1}^k H_t(x_i).
\end{align}
(where the functions on the target are smooth in the orbifold sense) 
and making a choice of $\mathfrak{b}, x$ as in Theorem \ref{t:spectral_intro},
we obtain a spectral invariant $$c_k( - \, ; M):=c(x,\Sym^k(-)): C^{\infty}([0,1]\times M) \rightarrow \R $$
which we usually denote by $c_k$ if the choice of the manifold $M, \mathfrak{b},  x$ are clear from the context.  The spectral invariant $c_k$ satisfies a raft of familiar properties.

\begin{theo}\label{thm:spec-sym-prod}
The symmetric product spectral invariant $c_k$ satisfy the following properties.

\begin{enumerate}
   \item (Symplectic invariance) $c_k(x,H\circ \psi) = c_k(x,H)$ for any $\psi \in \Symp(M, \omega)$;
    \item (Spectrality) for non-degenerate  $H$, $c_k(x,H)$ lies in the action spectrum $\Spec_k(H)$;
\item (Hofer Lipschitz) for any $H, H'$
$$ k  \int_{0}^1 \min  (H_t - H'_t) dt \leq c_k(x,H) - c_k(x,H') \leq k \int_{0}^1 \max\,  (H_t - H'_t) dt;$$
\item (Monotonicity) if $H_t \leq H'_t$ then $c_k(x,H) \leq c_k(x,H')$;
\item (Homotopy invariance) if $H, H'$ are mean-normalized and determine the same point of the universal cover $\widetilde{\Ham}(M, \omega)$, then $c_k(x,H) = c_k(x,H')$;
\item (Shift) $c_k(x,H + s(t)) = c_k(x,H) +  k \int_0^1 s(t)\, dt$;
\item (Subadditivity) for any $x,x' \in QH_{orb}(\Sym^k(M),\mathfrak{b})$ and $H,H'$, $c_k(x \cdot x',H \# H') \leq c_k(x,H) + c_k(x',H')$.
\end{enumerate}
\end{theo}

The action spectrum $\Spec_k(H)$, from item 2 above, is the $\Spec(Y,\Sym^k(H))$ from item 2 of Theorem \ref{t:spectral_intro} with respect to $\Pi_k$. To spell it out in terms of $M$ and $H$, let $x:S^1 \to Y$ be a generator of $CF(H_Y)$ as in \eqref{eq:orbit} and $\bar{c}_x:D^2 \to Y$ be an orbifold capping disc.
Let $\hat{x}:\hat{S} \to X=M^k$ and $\hat{c}: \hat{\Sigma} \to M^k$ be the equivariant maps corresponding to $x$ and $\bar{c}_x$ respectively.
Let $\pi_1:M^k \to M$ be the projection map to the first factor.
The composition $\pi_1 \circ \hat{c}:\hat{\Sigma} \to M$ factors through a degree $(k-1)!$ map $\hat{\Sigma} \to \tilde{\Sigma}$.
Denote the resulting map $\tilde{\Sigma} \to M$ by $\tilde{c}$.
Let $\tilde{S}:=\partial \tilde{\Sigma}$  and $\tilde{x}:=\tilde{c}|_{\partial \tilde{\Sigma}}$, where we equip $\partial \tilde{\Sigma}$ with the opposite of the boundary orientation due to the minus sign in the definition of a capping disc.
Note that $\tilde{x}$ is a collection of integral period orbits of $X_H$ such that the sum of the periods is $k$.
By definition,
\[
\int_x \Sym^k(H) dt=\frac{1}{k!} \int_{\hat{x}} H^{\oplus k} dt=\frac{1}{(k-1)!} \int_{\pi_1 \circ \hat{x}} H dt=\int_{\tilde{x}} H dt
\]
where $dt$ on $\hat{S}$ is defined to be the pull-back of the $dt$ on $S^1$, and the $dt$ on $\tilde{S}$ is the unique $1$-form whose pull-back to $\hat{S}$ is $dt$. 
Similarly, 
\[
\int_{\bar{c}_x} \omega_Y=\frac{1}{k!} \int_{\hat{c}} \omega_X=\frac{1}{(k-1)!} \int_{\pi_1 \circ \hat{c}} \omega_M=\int_{\tilde{c}} \omega_M
\]
The term $k_{(g)}(\bar{c}_x)\cdot \val(\frak{v}_{(g)})$ lies in $\Pi_k$.
Therefore, $\Spec(Y,\Sym^k(H))=\Spec_k(H)$ is given by
\begin{align}\label{eqn:sepc-defn}
    \Spec_k(H) =\Pi_k + \left\{ \int_{\tilde{x}} H dt - \int_{\tilde{c}} \omega_M \right\}.
\end{align}
where the second term runs over all possible $(\tilde{x},\tilde{c})$ such that $\tilde{x}$ is a collection of integral period orbits of $X_H$ whose sum of periods is $k$ and $\tilde{c}$ is a capping surface of $\tilde{x}$ as above.
A standard argument implies that $\Spec_k(H)$ is a Lebesgue measure-zero subset of $\R$.

\medskip
 The most crucial of the properties in Theorem \ref{thm:spec-sym-prod} are items 2,3,7; the other items (in slightly weaker forms) follow from these. We discuss this within the proof of Proposition \ref{l:char}.

\subsection{Sufficient conditions for Weyl laws}
Consider a sequence $c_k$ of spectral invariants, satisfying items 2,3,7 of Theorem \ref{thm:spec-sym-prod}.  We show in Proposition \ref{l:char} that if the associated {\it homogenized} spectral invariants  satisfy a Weyl law for characteristic functions of sufficiently small balls, then they will satisfy the Weyl law  for all Hamiltonians on $M$. 

Nonetheless, establishing the Weyl law in any given case seems challenging.  Our following result gives a sufficient condition.
\begin{prop}[see Proposition \ref{prop:idempotent}]\label{p:closedWeylInto}
Let $\mathfrak{b}_k \in QH^*_{orb}(\Sym^k(M);\Lambda_{>0})$ be a sequence of classes as in \eqref{eq:bulkb} for $k \in \mathbb{N}$.
Suppose that there exist $e_k \in QH^*(\Sym^k(M);\mathfrak{b}_k)$ such that
\begin{enumerate}
\item $e_k$ is an indecomposable idempotent, and $e_k \cdot QH^*(\Sym^k(M);\mathfrak{b}_k) = e_k\cdot \Lambda$ has rank one over the ground field;
\item the valuations $\val(e_k)$ grow sub-linearly with $k$, meaning $\val(e_k)/k \to 0$ as $k\to \infty$.
\end{enumerate}
Then the associated spectral invariants 
\begin{align*}
c_k: C^{\infty}([0,1]\times M) &\to \bR \text{ given by } H\mapsto c^{\mathfrak{b}_k}(e_k,\Sym^k(H))
\end{align*}
satisfy the Weyl law.
\end{prop}

In the case of the two-sphere $S^2$, we show that sequences $(\frak{b}_k,e_k)$ satisfying Proposition \ref{p:closedWeylInto}
exist by exploiting the interactions between Hamiltonian  and Lagrangian Floer homology; see Corollary \ref{cor:Weyl-S2}.  More precisely, we obtain the following result.

Let $U_{\Lambda} = \val^{-1}(1)$ denotes the unitary subgroup of the Novikov field.
\begin{prop}[see Proposition \ref{prop:links} and Corollary \ref{c:links}]
Let $L_k \subset M$ be a sequence of Lagrangian torus links of $k$ components, for $k=1,2,\ldots$. Choose the discrete groups $\Pi_k$ as in Remark \ref{rmk:choice of Pi}. Assume that there are bulk classes $\mathfrak{b}_k$ as above such that
\begin{enumerate}
\item The potential function $W^{\mathfrak{b}_k}$ of $L_k$ has an isolated non-degenerate critical point at $b_k \in H^1(\Sym(L_k);U_{\Lambda})$;
\item The valuation of the determinant of the Hessian of $W^{\mathfrak{b}_k}$ at the critical point $\frak{b}_k$ grows sublinearly with $k$.
\end{enumerate}
Then the image of the length zero open-closed map $\mathcal{OC}^0_{\Sym{L}_k}$ has rank one and contains  a unique idempotent $e_k \in QH^*(\Sym^k(M),\mathfrak{b}_k)$, and the  $\{e_k\}$ satisfy the hypotheses of Proposition \ref{p:closedWeylInto}.  Hence, the associated spectral invariants satisfy the Weyl law.
\end{prop}

Floer cohomology is always a deformation of ordinary cohomology. We restrict to `torus links', ones for which every component of $L_k$ (and hence also $\Sym(L_k)$) is a torus, with exterior algebra cohomology. 
A key point of the above result is that the associated symmetric product link has Floer cohomology a Clifford algebra. For general reasons (coming from the Calabi-Yau structure on Floer cohomology) this yields estimates on the valuations of quantum cohomology idempotents associated to non-degenerate critical points of the disc potential.

This pushes the eventual work onto the existence of a sequence of Lagrangian links $L_k \subset M$ of $k$ components, for which the orbifold Lagrangian Floer homology of the symmetric product links $\Sym(L_k)$ satisfies the above specific properties.  These in particular force the links to be `dense', which accounts for how they can probe the small-scale geometry of the manifold (essential for having a Weyl law).

On $S^2$ we prove that links of parallel circles satisfy the required conditions by adapting the arguments in \cite{MS19, Pol-Shel21} (see Lemma \ref{l:parallelLink}). By contrast, links composed of unions of toric fibres in $(S^2\times S^2, \omega + t\cdot\omega)$ never satisfy our conditions; more generally one can show that links which are products of links in the two factors and have non-zero orbifold Floer cohomology never become dense.

In general we need a larger supply of constructions of symmetric product links which might be Floer-theoretically unobstructed.

\begin{remark}
    If $L_k$ is a Floer-theoretically non-trivial link of $k$ components, the arguments of \cite{Pol-Shel21} show that the individual components have `finite packing' (the components admit at most $k$ pairwise disjoint Hamiltonian translates). By contrast, close cousins of the known links on $S^2 \times S^2$ admit infinitely many pairwise disjoint Hamiltonian translates \cite{Brendel-Kim}. This underscores the delicacy in finding suitable links.
\end{remark}

\begin{remark}
   Proposition \ref{prop:links} provides a criterion under which  Lagrangian/Hamiltonian spectral invariants become quasimorphisms. Applying the proposition to Lagrangian links in toric $4$-manifolds, such as those considered in \cite{MS19, Pol-Shel21}, one obtains new families of quasimorphisms on (the universal cover of) Hamiltonian diffeomorphisms.  This implies, in particular, that the invariants $c_{k, B}, \mu_{k,B}$, introduced in \cite{Pol-Shel21}, are quasimorphisms. See Remark \ref{rem:quasi-morphisms}.
\end{remark}

\subsection{Structure of the paper}\label{sec:structure}
Section \ref{Sec:template} discusses Weyl laws and their applications, such as the smooth closing lemma, assuming the existence of bulk-deformed orbifold Hamiltonian Floer theory with the `usual' properties. The companion paper \cite{MSS} gives a construction of that theory. The final sections outline our requirements from orbifold Lagrangian Floer theory. Most of these have been established in prior work of Cho and Poddar \cite{Cho-Poddar}. The one or two missing ingredients concern the open-closed map and Calabi-Yau structures, so necessarily entwining Lagrangian and Hamiltonian Floer cohomologies. However, \cite{Cho-Poddar} worked with the `Kuranishi atlas' approach to Floer theory, rather than the global charts used in \cite{MSS}. Given that, we have chosen to axiomatise the one or two remaining requirements of the orbifold Lagrangian theory that we require: these are summarised and then deployed in Section \ref{Sec:Lag Floer}.
\medskip

\noindent \emph{\textbf{Acknowledgments.}} 
We are grateful to Dan Cristofaro-Gardiner and Vincent Humili\`ere for our previous collaboration on Lagrangian links and dynamics, which motivated this work.
We thank Michael Hutchings and the UC Berkeley Department of Mathematics for their warm hospitality in December 2022, when this project was initiated.

C.Y.M is partially supported by the Royal Society University Research Fellowship.
I.S. is partially supported by UKRI Frontier Research Grant EP/X030660/1 (in lieu of an ERC Advanced Grant).
S.S. is partially supported by ERC Starting Grant number 851701.

\section{Weyl laws, idempotents, links and periodic orbits}\label{Sec:template}
We introduce here some of our notations and conventions. 

Let $(M, \omega)$ be a closed and connected symplectic manifold of dimension $2n$.   
Let $H : [0,1] \times M \to \R$ be a time-dependent Hamiltonian.  The Hamiltonian vector field $X_H$ of $H$ is defined by the identity
$$\omega(X_H, \cdot) :=-dH.$$

The Hamiltonian flow of $H$, that is the flow (or isotopy) generated by $X_H$,  is denoted by $\varphi_H^t$.  Its time-one map is written as  $\varphi_H$. The collection of all such time-one maps forms the group of Hamiltonian diffeomorphisms $\Ham(M,\omega)$ of $(M, \omega)$, which is a subgroup of the symplectomorphisms $\Symp(M, \omega)$.

There is a surjection
\begin{align*}
C^{\infty}([0,1]\times M;\bR) \longrightarrow \widetilde{\mathrm{Ham}}(M) 
\end{align*}
to the universal cover of the Hamiltonian group, which takes a time-dependent function $H_t$ to the homotopy class rel endpoints of the path $t \mapsto \varphi^t_H$, where $t \in [0,1]$; we denote the image of $H$ under this mapping by $\tilde{\varphi}_H$.  This map  remains a surjection when restricted to the subspace of mean-normalized Hamiltonians, i.e.\ $H_t$ such that $\int_M H_t \, \omega^n = 0$.

\subsection{Weyl laws and their applications} \label{sec:Weyl law applications}

\emph{In this section we formulate the Weyl law, and note that  the Weyl law implies both the smooth closing lemma and the non-simplicity of the group of Hamiltonian homeomorphisms.}

Let $(M,\omega)$ be a compact symplectic manifold of real dimension $2n$. Fix a discrete subgroup $\Pi_k$ for each $k=1,2,\dots$ to define the associated $\Spec_k(H)$ (see Equation \eqref{eqn:sepc-defn}).

A function $c: C^{\infty}([0,1]\times M,\bR) \rightarrow \bR$ is \emph{spectral} if it satisfies 
\begin{enumerate}
\item $c$ is continuous in the $C^\infty$ topology;
\item $c(H)$ belongs to the Lebesgue measure-zero subset $\cup_k \Spec_k(H)$.
\end{enumerate}

We say that a spectral function $c$ is {\em sub-additive} if, in addition to the above, it satisfies the triangle inequality
\begin{enumerate}
\item[3.] $c(H \# G) \leq c(H) + c(G)$.
\end{enumerate}

If $c$ is sub-additive, then, following Entov-Polterovich \cite{Entov-Polterovich}, we define the homogenized spectral function $\zeta : C^\infty([0,1] \times M) \rightarrow \R$ via
$$\zeta(H) := \lim_{m \to \infty} \frac{c(H^m)}{m},$$

where $H^m$ stands for the $m$-fold composition $H\# \ldots \# H$.  Sub-additivity guarantees the existence of the above limit.  Note that $\zeta$ is not necessarily spectral.

The prototypical example of a spectral function is the Floer theoretic spectral invariant associated to a choice of a quantum cohomology class $e \in QH^*(M)$, or more generally a class $e\in QH^*(M;\mathfrak{b})$ for some bulk deformation parameter $\mathfrak{b} \in H^{ev}(M;\Lambda_{>0})$. If the quantum cohomology class $e$ is an idempotent, then, the spectral invariant is sub-additive and so one can define its homogenization.  All of the spectral functions, and their homogenizations, which appear in this paper arise from idempotents in (bulk deformed) $QH$.

\medskip

    We say a sequence $c_k: C^{\infty}([0,1]\times M) \rightarrow \bR$ of spectral functions \emph{satisfies the Weyl law} if, the sequence $\frac 1k c_k(H)$  converges to the average of $H$, as in \eqref{eqn:Weyl law}.
   Similarly, we say a sequence $\zeta_k$ of homogenized spectral functions satisfies the Weyl law if for every Hamiltonian $H \in C^{\infty}([0,1]\times M)$, the sequence $\frac 1k\zeta_k(H)$  converges to the average of $H$.

It seems that homogenized spectral invariants have properties which are better suited for proving Weyl laws; see, for example, Proposition \ref{l:char}. Additionally, for applications, a Weyl law for $\zeta_k$ is just as effective as one for $c_k$.

\subsection*{Smooth closing lemmas \& non-simplicity}
We state here two potential applications of Weyl laws, which serve as motivations for our work. 

\subsubsection*{Smooth closing lemmas}
We say that \emph{the $C^{\infty}$-closing lemma holds} for Hamiltonian symplectomorphisms of  $(M,\omega)$ if, for any $\phi \in \mathrm{Ham}(M)$, any open $U\subset M$,  and any $C^{\infty}$-open neighborhood $\phi \in V_{\phi} \subset \mathrm{Ham}(M)$, there is a diffeomorphism $\phi' \in V_{\phi}$ which has a periodic point in $U$.  This is one of the most significant open problems in dynamical systems.  It has been resolved in the particular case of area-preserving surface diffeomorphisms \cite{airie, EH21, CPZ21} and 3D Reeb flows \cite{Irie}, as a corollary of Weyl laws in Periodic Floer and Embedded Contact homologies.

The smooth closing lemma remains wide open in higher dimensions, although there has been some  recent progress in very specific settings \cite{Cineli-Seyfaddini, CDPT, Xue}.  The following lemma explains how it can be deduced from Weyl laws.
\begin{lemma}\label{l:closing}
    If $M$ admits either a sequence of spectral functions $c_k$, or a sequence of homogenized spectral functions $\zeta_k$, which satisfy the Weyl law, then the $C^{\infty}$-closing lemma holds for Hamiltonian symplectomorphisms of $M$.
\end{lemma}
  The proof of this lemma follows Irie's argument from \cite{Irie}.
\begin{proof}
    Following Irie \cite{Irie24} (see also \cite{Cineli-Seyfaddini}), we say that a Hamiltonian diffeomorphism $\psi$ of $M$ satisfies the \emph{strong closing property} if for every non-zero time-dependent Hamiltonian function $G \geq 0$, there is some $s \in [0,1]$ for which $\varphi_G^s \circ \psi$ admits a periodic orbit which meets the support of $G$.  On the one hand, it is straightforward to see (by taking $G$ supported in a small ball) that if every $\psi$ satisfies strong closing, then $M$ satisfies the conclusions of the smooth closing lemma. On the other hand, the strong closing property for $\psi$ follows from the Weyl law. Indeed, take a Hamiltonian $H$ such that $\psi = \varphi_H$.  If the periodic orbits of $\varphi_G^s \circ \psi$ remain away from the support of $G$ for all $s \in [0,1]$, then for all $k, m \in \bN$ the action spectrum  $\Spec_k( (sG \# H)^m) $ remains independent of $s$, coinciding with $\Spec_k( (sG \# H)^m) $. For fixed $k, m$, the spectrality of the function $s \mapsto c_k((sG \# H)^m)$ forces it to be constant, since it is a continuous function taking values in the measure-zero set $\Spec_k(H^m)$.  Since this holds for all $k, m$, we conclude that the function $s \mapsto \zeta_k((sG \# H)^m)$ is also independent of $s$.  Hence, we have
    \begin{align*}
        c_k(G \# H) = c_k(H) \text{ and } \zeta_k(G \# H) = \zeta_k(H)
    \end{align*}
    for all $k \in \bN$.
    
    Since $G \geq 0$, if $G \neq 0$ then the average of $G\# H$ differs from the average of $H$, yielding a contradiction to the Weyl law for either of $c_k$ or $\zeta_k$.
\end{proof}

\subsubsection*{Non-simplicity of symplectic homeomorphisms}
Fathi proved in \cite{fathi} that the group of compactly supported volume-preserving homeomorphisms of the $n$-dimensional ball, for $n \geq 3$, is simple. In contrast, the group of compactly supported area-preserving homeomorphisms of the two-dimensional disc is not simple, as shown in \cite{CGHS20} (sometimes referred to as the simplicity conjecture \cite[Problem 42]{McDuff-Salamon}). Generalizations to other surfaces were treated in \cite{CGHMSS}.  

From a symplectic perspective, the natural higher-dimensional analogue of area-preserving homeomorphisms is the class of {\bf symplectic homeomorphisms}, defined as homeomorphisms of a symplectic manifold that arise as $C^0$ limits of symplectomorphisms. In two dimensions, these coincide with area-preserving homeomorphisms. Given the non-simplicity of the latter for the disc, a natural question arises: is the group of compactly supported symplectic homeomorphisms of the $2n$-dimensional ball simple?  

For area-preserving homeomorphisms, Weyl laws played a crucial in proving non-simplicity. Extending such results to higher dimensions faces the obstacle of lacking higher-dimensional Weyl laws. However, most, but not all, other aspects of the proof, rooted in $C^0$ symplectic geometry, do generalize. For instance, if Weyl laws for the spectral invariants $c_k$ from Hamiltonian Floer homology on $\Sym^k$ hold on a closed symplectic $(M, \omega)$, then the symplectic homeomorphisms of the $2n$-dimensional ball would also be non-simple. The reason for this is as follows. By embedding the ball into $(M, \omega)$ one would obtain spectral invariants $$c_k: \Ham(B^{2n}, \omega_0) \rightarrow \R.$$
 One can show that, on the one hand, the $c_k$ are continuous with respect to the $C^0$ topology and, on the other hand, their asymptotics recover the Calabi homomorphism which is defined as follows:
\begin{align*}
    \Cal :\Ham(B^{2n}, \omega_0) \rightarrow \R \\
    \varphi_H \mapsto \int_0^1 \int_{B^{2n}} H \omega^n \; dt.
\end{align*} 

Arguing as in \cite{CGHS20, CGHMSS}, one can use the above to construct an extension of the Calabi homomorphism to the group $\mathrm{Hameo}(B^{2n}, \omega_0)$; this is the normal subgroup of symplectic homeomorphisms, introduced by Müller and Oh \cite{muller-oh}, consisting of those symplectic homeomorphisms $\varphi$ which can be written as the $C^0$ limit of a sequence $\varphi_i \in \Ham(B^{2n}, \omega_0)$ such that the sequence $\varphi_i$ is Cauchy in Hofer's distance.  Hence, $\mathrm{Hameo}(B^{2n}, \omega_0)$ is not simple, as the kernel of the extension of Calabi is a proper normal subgroup. From this, we deduce that symplectic homeomorphisms must be non-simple, as either $\mathrm{Hameo}$ or the kernel of the Calabi homomorphism forms a proper normal subgroup.

\medskip
We therefore focus attention on Weyl laws.

\subsection{Weyl laws \& characteristic functions}\label{sec:Weyl-char}

\emph{In this section we show that to prove the Weyl law, it suffices to prove that it holds for characteristic functions of small sets.}

Consider a sequence of sub-additive spectral invariants $c_k: C^\infty([0,1] \times M) \rightarrow \R$ which satisfy the following version of spectrality
\begin{equation}\label{eqn:spectrality-symprod}
 c_k(H) \in \Spec_k(H).
\end{equation}

The spectral invariants arising from Hamiltonian Floer homology on $\Sym^k(M)$ have this property; see Theorem \ref{thm:spec-sym-prod}.  The main goal of this section is to show that, in the case of such spectral invariants, to establish the Weyl law it suffices to prove it for a very restricted class of functions, namely characteristic functions of sufficiently small balls.

Consider a Darboux coordinate chart $\mathcal U \subset M$ with coordinates $x_1, y_1, \ldots, x_n, y_n$.  Let $B_a \subset \mathcal U$ be the ball of radius $a$ in these coordinates and denote by $\chi_a$ the characteristic function of $B_a$.   As we explain below, since $c_k$ and $\zeta_k$ are monotone, one can define $c_k(\chi_a)$ and $\zeta_k(\chi_a)$, for all $k$ and $a$.

\begin{prop}\label{l:char}
    Suppose that for all sufficiently small $a > 0$, we have $$\lim_k \frac{1}{k} \zeta_k(\pm \chi_a) = \pm \frac{\mathrm{Vol}(B_a)}{\mathrm{Vol}(M)}.$$
     Then,  $$\lim_k \frac{1}{k} \zeta_k(F) = \frac{1} 
 {\mathrm{Vol}(M)} \int_0^1 \int_M F \, \omega^n$$ for all Hamiltonians $F$.
\end{prop}
The rest of this section is dedicated to the proof of this proposition.
\begin{proof}
We have a sequence $c_k$ of spectral functions which satisfy, in addition to the three properties listed at the beginning of this section, the form of spectrality in \eqref{eqn:spectrality-symprod}.  Arguing as in \cite[Section 3]{HLS1}, one can show that the $c_k$ must also satisfy the following additional properties.
   \begin{itemize}
       \item Hofer Lipschitz: for any $H, H'$
$$  k\int_{0}^1 \min  (H_t - H'_t) dt \leq c_k(H) - c_k(H') \leq  k \int_{0}^1 \max\,  (H_t - H'_t) dt;$$
       \item Symplectic invariance: $c_k(H \circ \psi) = c_k(H)$ if $\psi \in \mathrm{Symp}_0(M, \omega)$, where $\mathrm{Symp}_0$ denotes the connected component of $\mathrm{Id}$ in $\mathrm{Symp}$;
       \item Monotonicity: if $H_t \leq H'_t$ then $c(x,H) \leq c(x,H')$;
       \item Homotopy invariance: if $H, H'$ are mean-normalized and determine the same point of the universal cover $\widetilde{\Ham}(M, \omega)$, then $c_k(H) = c_k(H')$;
        \item Shift: $c_k(H + s(t)) = c_k(H) + k \int_0^1 s(t) dt$.
   \end{itemize}

  We will not prove the above as it can be proven by adapting the arguments in \cite[Section 3]{HLS1}.  Moreover, in the case where the $c_k$ are obtained from orbifold Hamiltonian Floer homology on $\Sym^k(M)$, these properties are established in Theorem \ref{thm:spec-sym-prod}, as an immediate consequence of Theorem \ref{t:spectral_intro}.

  {\bf \noindent Spectral invariants for characteristic functions.} 
  We begin by  defining $c_k(\chi_a)$ and $\zeta_k(\chi_a)$ for the characteristic function $\chi_a$ of $B_a$.   Let $\chi^r_a : M \to \bR$ be a smooth function which is supported in the slightly larger ball $B_{a+r} \subset \mathcal{U}$, depends only on the radius (meaning $\chi^r_a(x_1, y_1, \ldots, x_n, y_n)$ depends only on the value of $\sum x_i^2 + y_i^2$),  takes the the value $1$ on $B_a$, and is non-increasing as a function of the radius.  We, moreover, pick these functions such that $\chi^{r_1}_a \leq \chi^{r_2}_a$ if $r_1 \leq r_2$.  We define $$c_k(\chi_a):= \lim_{r \to 0} c_k(\chi^r_a).$$
  The limit exists because the function $r \mapsto c_k(\chi^r_a)$ is decreasing and bounded below (by $c_k(F)$ for any $F \leq \chi_a)$.  One can similarly define $c_k(c \,  \chi_a)$ for any constant $ c \in \R$. 

  As for $\zeta_k( \chi_a)$ it can be defined in two different manners with the end results coinciding.  One way would be to define $\zeta_k(\chi_a):= \lim_{r \to 0} \zeta_k(\chi^r_a)$.  A second way would be to define  it as the homogenization
  $\zeta_k(\chi_a) := \lim_{m \to \infty} \frac{1}{m} c_k(m \, \chi_a)$.  We leave it to the reader to verify that the two definitions coincide.  We have the following useful property for $\zeta_k$: for any $c\geq 0$ we have 
  \begin{align}\label{eqn:zeta-linear}
      \zeta_k(c \, \chi_a) = c\,  \zeta_k(\chi_a), \zeta_k(- c \, \chi_a) = c \, \zeta_k( - \chi_a).
  \end{align}

  Note also that one can similarly define  $c_k$ and $\zeta_k$ for (scalar multiples of) characteristic functions of sets of the form $\psi(B_a)$ with $\psi \in \Symp_0(M, \omega)$, i.e. functions of the form  $ c \, \chi_a \circ \psi^{-1}$. This can be done by working in the coordinate chart $\psi^{-1}(\mathcal{U})$ and the corresponding Darboux coordinates $x_i \circ \psi^{-1}, y_i \circ \psi^{-1}$.  By the symplectic invariance property we have 
  \begin{align*} 
c_k( c \,  \chi_a \circ \psi^{-1} ):= c_k(c \,  \chi_a), \, \zeta_k(c \, \chi_a \circ \psi^{-1}) := \zeta_k( c \,  \chi_a).
\end{align*}

Finally, suppose that $\psi_1, \ldots, \psi_N \in \mathrm{Symp}_0$ are such that the sets $\psi_i(B_{a})$ are pairwise disjoint.  Then, as above, we can define $c_k$, $\zeta_k$ for a function of the form $ \sum_{i=1}^N m_i \, \chi_a \circ \psi_i^{-1} $, where $m_i \in \bR$ are constants.  The definition, in the case of $c_k,$ is given by the formula:  
   $$ c_k( \sum_{i=1}^N m_i \, \chi_a \circ \psi_i^{-1}) := \lim_{r \to 0 } c_k( \sum_{i=1}^N m_i \, \chi_a \circ \psi_i^{-1} ),$$
   with a similar definition for the $\zeta_k$.

Suppose that $\psi_1, \ldots, \psi_N \in \Ham(M, \omega)$. Then, the  $\zeta_k$ satisfy the following property. 
\begin{claim}\label{cl:Sikorav2}
Let \[ G' = \sum_{i=1}^N m_i \, \chi_a, \; G = \sum_{i=1}^N m_i \, \chi_a \circ \psi_i^{-1}, \]
where the $m_i$ are constants and the $\psi_i \in \Ham(M,\omega)$ are such that the symplectic balls $\psi_1(B_a), \ldots, \psi_N(B_a)$ are pairwise disjoint.  Then,  $ \zeta_k(G) = \zeta_k(G').$

\end{claim}
\begin{proof}[Proof of Claim \ref{cl:Sikorav2}]

For small $r >0$, define $G'_r = \sum_{i=1}^N m_i \, \chi^r_a, \; G_r = \sum_{i=1}^N m_i \, \chi^r_a \circ \psi_i^{-1}$.  We will show that $ \zeta_k(G_r) = \zeta_k(G'_r)$, which suffices because $ \zeta_k(G) = \lim_{r \to 0 }\zeta_k(G_r)$ and $\zeta_k(G') = \lim_{r \to 0 } \zeta_k(G'_r)$, by definition.

Let $F_1, F_2$ be any two mean-normalized (smooth) Hamiltonians. The Hofer Lipschitz and homotopy invariance properties of the $c_k$, imply that   $$|c_k(F_1) - c_k(F_2) |\leq k \, \tilde{d}_H (\tilde{\varphi}_{F_1}, \tilde{\varphi}_{F_2}),$$
where $\tilde{d}_H$ denotes Hofer's (pseudo) distance on the universal cover $\widetilde{\Ham}$, and $\tilde{\varphi}_{F_1}, \tilde{\varphi}_{F_2}$ are the elements determined by $F_1, F_2$ in $\widetilde{\Ham}$.  The Hofer distance on the universal cover $\widetilde{\Ham}$ is defined via the formula
$$ \tilde{d}_H(\tilde{\psi}_1, \tilde{\psi}_2):= \inf \{\Vert H \Vert_{(1, \infty)} : \tilde\varphi_H = \tilde{\psi}_1 \tilde{\psi}_2^{-1} \}, $$
where $$ \Vert H \Vert_{(1, \infty)} := \int_0^1 \left( \max_M H(t, \cdot) - \min_M H(t, \cdot) \right) \, dt. $$
It is bi-invariant, however, its non-degeneracy, which is not needed for our arguments, is unknown.\footnote{The Hofer distance on $\Ham$, defined via 
$d_H(\psi_1, \psi_2):= \inf \{ \Vert H \Vert_{(1, \infty)} : \varphi_H = \psi_1\psi_2^{-1} \}$, is non-degenerate.}

It follows from the previous inequality that $$|c_k(F_1^n) - c_k(F_2^n) |\leq k \, \tilde{d}_H (\tilde{\varphi}_{F_1}^n, \tilde{\varphi}_{F_2}^n),$$
for all integers $n$.  Hence, if $F_1, F_2$ are such that $\tilde{d}_H (\tilde{\varphi}_{F_1}^n, \tilde{\varphi}_{F_2}^n)$ is bounded independently of $n$, then $\zeta_k(F_1) = \zeta_k(F_2)  $.

We will show below that $\tilde{d}_H (\tilde{\varphi}_{G_r'}^n , \tilde{\varphi}_{G_r}^n)$ is bounded independently of $n$. This, on its own, does not imply $\zeta_k(G_r') = \zeta_k(G_r)$ because $G_r'$ and $G_r$ are not mean-normalized. However, they have the same average, i.e. $\int G_r' \, \omega^n  =  \int G_r \, \omega^n$, and this suffices for concluding that $\zeta_k(G_r') = \zeta_k(G_r)$. 

 Note the $\varphi^t_{G'_r} = \varphi^t_{m_1 \chi_a^r} \circ \ldots \circ \varphi^t_{m_N \chi_a^r}$ and $\varphi^t_{G_r} = \varphi^t_{m_1 \chi_a^r \circ \psi_1^{-1}} \circ \ldots \circ \varphi^t_{m_N \chi_a^r \circ \psi_N^{-1}} = \psi_1 \varphi^t_{m_1 \chi_a^r} \psi_1^{-1} \circ \ldots \circ \psi_N \varphi^t_{m_N \chi_a^r} \psi_N^{-1}$.  Hence, we can write $\tilde{\varphi}_{G'_r} = \tilde{\varphi}_{m_1 \chi_a^r} \circ \ldots \circ \tilde{\varphi}_{m_N \chi_a^r}$ and $\tilde{\varphi}_{G_r} = \tilde{\psi}_1\tilde{\varphi}_{m_1 \chi_a^r} \tilde{\psi}_1^{-1} \circ \ldots \circ \tilde{\psi}_N\tilde{\varphi}_{m_N \chi_a^r} \tilde{\psi}_N^{-1},$ where $\tilde{\psi}_i$ is a lift of $\psi_i$ to $\widetilde{\Ham}(M)$.
 
We have the identity
$$\tilde{\varphi}_{G'_r}^n = \tilde{\varphi}^n_{m_1 \chi_a^r} \circ \ldots \circ \tilde{\varphi}^n_{m_N \chi_a^r}$$
for all $n \in \N$, because the elements involved in the composition pairwise commute. Similarly, we have  $$\tilde{\varphi}_{G_r}^n =  \tilde{\psi}_1\tilde{\varphi}^n_{m_1 \chi_a^r} \tilde{\psi}_1^{-1} \circ \ldots \circ \tilde{\psi}_N\tilde{\varphi}^n_{m_N \chi_a^r} \tilde{\psi}_N^{-1}$$
for  all $n \in \N$. Here, we are using the fact the $\tilde{\psi}_i\tilde{\varphi}_{m_i \chi_a^r} \tilde{\psi}_i^{-1}$ pairwise commute which holds because the generating Hamiltonians $\chi^r_a \circ \psi_i^{-1}$ have pairwise disjoint supports.

Consequently, 
\begin{align*}
\tilde{d}_H (\tilde{\varphi}_{G_r'}^n , \tilde{\varphi}_{G_r}^n) & = \tilde{d}_H (\tilde{\varphi}^n_{m_1 \chi_a^r} \circ \ldots \circ \tilde{\varphi}^n_{m_N \chi_a^r}, \tilde{\psi}_1\tilde{\varphi}^n_{m_1 \chi_a^r} \tilde{\psi}_1^{-1} \circ \ldots \circ \tilde{\psi}_N\tilde{\varphi}^n_{m_N \chi_a^r} \tilde{\psi}_N^{-1}) 
 \\ 
 & \leq \sum_{i=1}^N \tilde{d}_H(\tilde{\varphi}^n_{m_i \chi_a^r},  \tilde{\psi}_i\tilde{\varphi}^n_{m_i \chi_a^r} \tilde{\psi}^{-1}_i) \\
& \leq \sum_{i=1}^N 2 \tilde{d}_H(\tilde{\psi}_i, \id). 
\end{align*}
The RHS of the above inequality is independent of $N$.  This completes the proof of Claim \ref{cl:Sikorav2}.  
\end{proof}

We can now prove that the Weyl law for characteristic functions implies the Weyl law for any $F \in C^\infty([0,1] \times M)$, that is 
  
  \begin{equation}\label{eqn:Weyl law}
\lim_{k \to \infty } \; \; \frac{1}{k} \zeta_k(F) = \frac{1}{\mathrm{Vol}( M)}\int_0^1 \int_{ M} F_t \,  \omega^n\, dt.
\end{equation}

\medskip

To simplify our notation, for the rest of the proof we will denote  $\int F : = \int_0^1 \int_{ M} F_t \,  \omega^n\, dt $ for any Hamiltonian $F$ on $ M$.

  {\noindent \bf The lower bound.}  We will now prove that
\begin{align}\label{eqn:Weyl-lowerbound}
\liminf_{k \to \infty } \; \; \frac{1}{k} \zeta_k(F) \geq \frac{1}{\mathrm{Vol}(M)}\int F.
\end{align}

According to the Shift property $\zeta_k(F + c) = \zeta_k(F) + kc$  for any constant $c$.  Hence, up to the addition of a constant,  we may assume without loss of generality that $F$ is strictly positive.

We will need the following lemma whose proof, which we leave to the reader, follows from standard arguments in symplectic packing and Riemann integration theory.
\begin{lemma}\label{lem:Riemann-integration} Let $F_1, \ldots, F_p$ be time independent Hamiltonians.  For any given $\varepsilon >0$, there exists $a > 0$, Hamiltonian diffeomorphisms $\psi_1, \ldots, \psi_N \in \Ham ( M)$ with the following properties:
\begin{enumerate}
\item  $\psi_i(B_{a})$ are pairwise disjoint, and moreover $$\mathrm{Vol} ( \sqcup_{i=1}^N \psi_i(B_{a} )) \geq \mathrm{Vol} ( M) - \varepsilon, $$

\item For $i=1, \ldots, N$ and $l=1, \ldots, p$, let $$m_i(F_l) = \inf_{\psi_i(B_a)}F_l.$$
The functions $G_1, \ldots, G_p$ defined by the formula
    $$G_l:= \sum_{i=1}^N m_i(F_l) \chi_a \circ \psi_i^{-1}$$ satisfy the inequalities $G_l \leq F_l$ and
        \begin{align} \label{eqn:Riemann-sum-estimate}
       \frac{1}{\mathrm{Vol}( M) } \int F _l -\varepsilon \leq \frac{1}{\mathrm{Vol} M) }  \int G_l.
        \end{align}  
\end{enumerate}
\end{lemma}

Proceeding with the proof of the lower bound inequality \eqref{eqn:Weyl-lowerbound}, we will first prove it under the assumption that the Hamiltonian $F$  is time independent.   Fix $\eps >0$ and apply Lemma \ref{lem:Riemann-integration} to $F$ to obtain $a >0$,  $\psi_1, \cdots, \psi_N \in \Ham(M)$, and $G$ as in the conclusion of Lemma \ref{lem:Riemann-integration}; note that in this case we have $p=1$.   The Weyl law for $\chi_a$, which we are assuming, states that  
\begin{align}\label{eqn:stable-Weyl1}
\frac{1}{k} \zeta_k( \chi_a ) \to \frac{ \mathrm{Vol}(B_a)}{\mathrm{Vol}( M)}
\end{align}

Observe that, by \eqref{eqn:zeta-linear}, we have $\zeta_k(\sum m_i(F) \, \chi_a ) =  \sum m_i(F) \zeta_k(\chi_a)$; note $m_i(F) \geq 0$. 
 It follows from \eqref{eqn:stable-Weyl1}, combined with the fact that $\zeta_k(\sum m_i(F) \, \chi^r ) = \sum m_i(F) \zeta_k(\chi^r)$, that 
\begin{align*}  
 \frac{1}{k} \zeta_k( \sum_{i=1}^N m_i(F) \, \chi_a ) \to \frac{1}{\mathrm{Vol}( M)} \sum_{i=1}^N m_i(F) \, \mathrm{Vol}(B_a).  
\end{align*}

Now, since $G = \sum_{i=1}^N m_i(F) \chi_a \circ \psi_i^{-1}$, we have $\int G = \sum_{i=1}^N m_i(F) \, \int \chi_a$   and so from the above we conclude that

$$  \frac{1}{k} \zeta_k( \sum_{i=1}^N m_i(F) \, \chi_a  ) \to \frac{1}{\mathrm{Vol}( M)} \int G \geq \frac{1}{\mathrm{Vol}( M)} \int F -\varepsilon,$$

where the last inequality follows from \eqref{eqn:Riemann-sum-estimate}.  According to Claim \ref{cl:Sikorav2}, we have that $\zeta_k(G) = \zeta_k( \sum_{i=1}^N m_i(F) \, \chi_a)$.  Hence the above becomes 
$$  \frac{1}{k} \zeta_k( G ) \to \frac{\int G}{\mathrm{Vol}( M)} \geq \frac{\int F}{\mathrm{Vol}( M)} -\varepsilon,$$
 Noting that $\zeta_k(G) \leq \zeta_k(F)$, because $G \leq F$ and $\zeta_k$ is monotone, we conclude that for all $k$ large we have 
$$\frac{1}{k} \zeta_k(F) \geq  \frac{\int F}{\mathrm{Vol}( M)} - \varepsilon,$$
which establishes the lower bound inequality \eqref{eqn:Weyl-lowerbound} for time independent $F$.

   Suppose now that $F$ is time-dependent.  
 
 \begin{claim}\label{cl:approx-timedep}
 For any given $\eta > 0$, there exists time-independent Hamiltonians $F_1, \ldots, F_p$ such that
$$ \Vert F - F_1 \diamond \cdots \diamond F_p  \Vert_{(1, \infty)} \leq \eta.$$

 \end{claim} 
 In the above $\diamond$ denotes the concatenation operation which we now define.  Suppose that we are given, possibly time dependent, Hamiltonians $H_1, \ldots, H_p: [0,1] \times M \rightarrow \R $.   Via a well-know reparametrization argument, we may assume that $H_i(t, \cdot) = 0$ for $t$ sufficiently close to $0,1$.  

Define 
\begin{equation*}\label{eq:sharp}
H_1 \diamond \cdots \diamond H_p(t, x) := pH_i(p (t -(i-1)/p), x) \quad \text{for}\quad t \in [(i-1)/p, i/p].
\end{equation*}
  The Hamiltonian isotopy generated by $H_1 \diamond \cdots \diamond H_p$ is simply (a reparametrization of) the concatenation of the Hamiltonian flows of $H_1, \ldots, H_p$. In particular, $\varphi_{H_1 \diamond \cdots \diamond H_p } = \varphi_{H_p} \circ \cdots \circ \varphi_{H_1}$.  Moreover, the two Hamiltonian isotopies  $\varphi^t_{H_1 \diamond \cdots \diamond H_p }$ and  $\varphi^t_{H_p} \circ \cdots \circ \varphi^t_{H_1}$ are homotopic rel endpoints.

\begin{proof}[Proof of Claim \ref{cl:approx-timedep}]
Pick $p$ so large so that $\Vert F(t, \cdot ) -  F(\frac{i}{p}, \cdot) \Vert_\infty \leq \eta$ for every $t \in [\frac{i-1}{p}, \frac{i}{p}]$, with $i=1, \ldots, p$. 
Let $F_i(\cdot) = \frac{1}{p} F(\frac{i}{p}, \cdot)$.
\end{proof} 
 
 As a consequence of Claim \ref{cl:approx-timedep}, it suffices to prove \eqref{eqn:Weyl-lowerbound} for a Hamiltonian of the form $F_1 \diamond \cdots \diamond F_p$ where each of $F_1, \ldots, F_p$ is time independent.

 Fix $\varepsilon >0$.  Apply Lemma \ref{lem:Riemann-integration} to the Hamiltonians $F_1, \ldots, F_p$, and let $a$, $\psi_1, \ldots, \psi_N$ and $G_1, \ldots, G_p$ be as in the conclusion of the lemma. For $i=1, \ldots, p$, let $G_l' = \sum_{i=1}^N m_i(F_l) \chi_a$.
 
Our argument from the time-independent case implies that for $k$ large we have 
\begin{equation}\label{eqn:time-indep-case}
 \frac{1}{k}\zeta_k(G_l) =\frac{1}{k}\zeta_k(G'_l) \geq \frac{\int F_l}{\mathrm{Vol}( M)} - \varepsilon.
\end{equation}

Since $G_l \leq F_l$, we have that $G_1 \diamond \ldots \diamond G_p \leq F_1 \diamond \ldots \diamond F_p$ and so we get the inequality 

\begin{equation}\label{eqn:G_l leq F_l}
\zeta_k(G_1 \diamond \ldots \diamond G_p) \leq \zeta_k(F_1 \diamond \ldots \diamond F_p).
\end{equation}

\begin{claim}\label{eqn:diamond=plus}
$\zeta_k(G_1 \diamond \ldots \diamond G_p) = \zeta_k(G_1 + \ldots + G_p).$
\end{claim}
\begin{proof}
As in Claim \ref{cl:Sikorav2}, one proves this by showing that $\zeta_k(G_{1,r} \diamond \ldots \diamond G_{p,r}) = \zeta_k(G_{1,r} + \ldots + G_{p,r})$
where $G_{i,r}$ is a smooth approximation of $G_i$.
Now, for all sufficiently small $r>0$, the $G_{i,r}$ consist of sums of commuting time-independent Hamiltonians, which then implies that $G_{1,r} \diamond \ldots \diamond G_{p,r}$ and $G_{1,r} + \ldots + G_{p,r}$ determine the same element in the universal cover $\widetilde{\Ham}$.  Moreover, they have the same average.  It then follows that $\zeta_k(G_{1,r} \diamond \ldots \diamond G_{p,r}) = \zeta_k(G_{1,r} + \ldots + G_{p,r})$.
\end{proof}

Now note that the time-independent Hamiltonians $G_1 + \ldots + G_p, G'_1 + \ldots + G'_p$ are of the same form as the Hamiltonians $G, G'$ in Claim \ref{cl:Sikorav2} and hence
\begin{equation}\label{eqn:G_l=G'_l}
\zeta_k(G_1 + \ldots + G_p) = \zeta_k(G'_1 + \ldots + G'_p).
\end{equation}

Since the $G'_l$ are all positive scalar multiples of $\chi_a$ we have, by \eqref{eqn:zeta-linear}, that $\zeta_k(G'_1 + \ldots + G'_p) = \zeta_k(G'_1) + \ldots + \zeta_k( G'_p)$.  Combining this with \eqref{eqn:G_l leq F_l}, \eqref{eqn:diamond=plus}, \eqref{eqn:G_l=G'_l} we obtain

$$\zeta_k(G'_1) + \ldots + \zeta_k(G'_p) \leq \zeta_k( F_1 \diamond \ldots \diamond F_p).$$

This inequality together with \eqref{eqn:time-indep-case}
gives us 

$$\sum_{l=1}^p \frac{\int F_l}{\mathrm{Vol}( M)} - p \varepsilon \leq \frac{1}{k} \zeta_k( F_1 \diamond \ldots \diamond F_p).$$

Since $\sum_{l=1}^p \int F_l = \int F_1 \diamond \ldots \diamond F_p$, we conclude that the \eqref{eqn:Weyl-lowerbound} holds for $F_1 \diamond \ldots \diamond F_p$.

 \medskip
{\noindent \bf The upper bound.} It remains to prove that
\begin{align}\label{eqn:Weyl-upperrbound}
\limsup_{k \to \infty } \; \; \frac{1}{k} \zeta_k(F) \leq \frac{1}{\mathrm{Vol}(M)}\int F.
\end{align}

The proof of this is very similar to that of the lower bound inequality \eqref{eqn:Weyl-lowerbound}, hence we will omit it.  We only remark here that the main difference between the two cases is that the proof of the lower bound inequality relied on the Weyl law for $\chi_a$ while the upper bound inequality requires the Weyl law for $-\chi_a$.
\end{proof}

\subsection{Idempotents and symmetric product orbifolds}

\emph{We formulate conditions on a sequence of idempotents associated to symmetric product orbifolds which are sufficient to infer the Weyl law. The argument uses the sequence of quasimorphisms associated to the idempotents.}

Even allowing for bulk deformations, the ring $QH^*(M)$ does not have enough interestingly different spectral invariants to possibly yield a Weyl law.  Following ideas from \cite{CGHMSS}, one obtains more flexibility by introducing the symmetric products of $M$, and using the existence of the canonical maps
\begin{align*}
C^{\infty}([0,1]\times M) &\to C^{\infty}([0,1] \times \Sym^k(M))\\
H & \mapsto \Sym^k(H)_t(x_1,\dots,x_k):=\sum_{i=1}^k H_t(x_i).
\end{align*}
where the target denotes the functions which are smooth in the orbifold sense (locally descended from smooth equivariant functions on uniformizing orbifold charts).

We will consider a sequence of classes $e_k \in QH^*(\Sym^k(M);\mathfrak{b}_k)$ where $\mathfrak{b}_k \in H^*(I\Sym^k(M);\Lambda_{>0})$ is a certain bulk deformation class.

\begin{prop}\label{prop:idempotent}
Suppose the $e_k$ satisfy that
\begin{enumerate}
\item $e_k$ is an indecomposable idempotent, and indeed $e_k \cdot QH^*(\Sym^k(M);\mathfrak{b}_k) = e_k\cdot \Lambda$ has rank one over the ground field;
\item the valuations $\val(e_k)$ grow sub-linearly with $k$, meaning $\val(e_k)/k \to 0$ as $k\to \infty$.
\end{enumerate}
Then the associated spectral invariants 
\begin{align*}
c_k: C^{\infty}([0,1]\times M) &\to \bR \text{ given by }H\mapsto c^{\mathfrak{b}_k}(e_k,\Sym^k(H))
\end{align*}
satisfy the Weyl law. 
\end{prop}

\begin{proof}[Proof of Proposition \ref{prop:idempotent}]  
The first hypothesis states that $QH^*(\Sym^k(M);\mathfrak{b}_k)$ splits, as an algebra, into a direct sum of two factors one of which is the ground field and $e_k$ is the identity in this factor.  The fact that the ground field appears as a summand combined with the proof of Theorem 12.6.1 in \cite{Polterovich-Rosen} implies that for any Hamiltonian $H$ we have the inequality  
$$c_k(H) + c_k(\bar{H}) \leq - \val(e_k).$$  

Let $\mu_k(\varphi_H):=\frac{1}{k}c_k(H)$ for mean-normalized $H$.
According to \cite[Prop.\ 5.1]{usher11}, from the above inequality we can conclude that $$\mu_k: \widetilde{\mathrm{Ham}}(M) \rightarrow \R$$ is a quasimorphism with defect at
 most $ -\frac{1}{k} \val(e_k)$.  
 
 Now, the sequence of quasimorphisms $\mu_k: \widetilde{\mathrm{Ham}}(M) \rightarrow \R$ satisfy the following properties:
 \begin{enumerate}
 \item The $\mu_k$ are pointwise bounded: $\forall g \in \widetilde{\mathrm{Ham}}(M)$ the sequence $\{\vert \mu_k(g) \vert:k \in \mathbb{N}\}$ is bounded by the Hofer Lipschitz property of the spectral invariants $c_k$.
 \item The defects of the $\mu_k$ converge to zero, by the hypothesis $\val(e_k)/k \to 0$ as $k\to \infty$. 
 \end{enumerate}
 
Since $\widetilde{\mathrm{Ham}}$ is perfect, it then follows from Lemma \ref{lem:quasi-converge-hom}, stated below, that the $\mu_k$ converge point-wise to zero, i.e.\ the $c_k$ satisfy the Weyl law for mean-normalized Hamiltonians.  The Shift property of the $c_k$ implies that the Weyl law holds for all Hamiltonians.
\end{proof}

Lemma \ref{lem:quasi-converge-hom} tell us that a pointwise bounded sequence of quasimorphisms, with defects converging to zero, converges to a group homomorphism.  Here is the precise statement.

\begin{lemma}\label{lem:quasi-converge-hom} Let $\mu_k : G \rightarrow \R$ be a sequence of quasimorphisms on a group $G$.
Suppose that 
\begin{enumerate}
\item The $\mu_k$ are point-wise bounded, i.e.\  the sequence $\{ | \mu_k(g)|:k \in \mathbb{N} \}$ is bounded  $\forall g \in G$.
\item The defects of the $\mu_k$ converge to zero.
\item The group $G$ is perfect.
\end{enumerate}
Then, the $\mu_k$ converge point-wise to zero.
\end{lemma}
\begin{proof}
  We will show that for each $g \in G$, the limit 
  $\displaystyle \lim_k \mu_k(g)$ exists and equals zero.

  Fix $g \in G$ and let $\hat{\mu}(g)$ be any limit point of the sequence $\mu_k(g)$; note that $\hat{\mu}(g)$ exists by our assumption on point-wise boundedness of the $\mu_k$.  We will show that $\hat{\mu}(g)$ is zero, which clearly implies that $\displaystyle \lim_k \mu_k(g)$ exists and equals zero.

  Up to passing to a subsequence and using the facts that the defect of $\mu_k$ converges to $0$, we may assume that the sequences $\mu_k(g)$ and $\mu_k(g^{-1})$ converge to $\hat{\mu}(g)$ and $-\hat{\mu}(g)$, respectively.  Now, pick $g_1, \ldots, g_m$ such that $g$ can be written as a product of commutators of the form $[g_i, g_j]$; such $g_i$ exist by the assumption on perfectness of $G$.  The point-wise boundedness assumption on the $\mu_k$ implies that, we can find $k_l \to \infty$ such that all sequences $\mu_{k_l}(g_i)$ and $\mu_{k_l}(g_i^{-1})$ have limits for $i=1, \ldots, m$.  We will denote the limits by $\hat{\mu}(g_i^{\pm 1})$.

  Let $H$ be the normal subgroup in $G$ generated by the $g_i$.  Note that $[g_i, g_j] \in H$ and hence $g \in [H,H]$.

  \begin{claim}\label{cl:limit-exists}
      For every $h \in H$, the limit
      $$ \lim_{k_l \to \infty} \mu_{k_l}(h)$$
      exists.
  \end{claim}

Before proving Claim \ref{cl:limit-exists}, we explain why it implies the conclusion of Lemma \ref{lem:quasi-converge-hom}.  As a consequence of the claim we can define
$$\hat{\mu}: H \to \R$$ to be the point-wise limit of the $\mu_{k_l}$ (on $H$). Observe that we have
\begin{align*}
    |\hat{\mu}(h_1 h_2) -  \hat{\mu}(h_1) - \hat{\mu}(h_2)| &= | \lim_{k_l \to \infty }  \mu_{k_l}(h_1 h_2) -  \mu_{k_l}(h_1) - \mu_{k_l}(h_2) \; | \\
    &\leq \lim_{k_l \to \infty }  \mathrm{def}(\mu_{k_l}) =0,
    \end{align*}
    from which it follows that $\hat{\mu}: H \rightarrow \R $ is a group homomorphism.  Hence, $\hat{\mu}$ vanishes on $[H,H]$ and in particular  $\hat{\mu}(g) = 0$.

  \begin{proof}[Proof of Claim \ref{cl:limit-exists}]
      To simplify the notation within the proof of this claim, we will write $\mu_{k}$ instead of $\mu_{k_l}$
      
      Let $h$ denote an element of $H$.  We can write $$ h = \prod_{j=1}^N \ell_j f_j^{\pm 1} \ell_j^{-1} $$
      where $f_j \in \{g_1, \ldots, g_m\}$.  Note that we have $$\lim_{k \to \infty} \mu_k(f_j) = \hat{\mu}(f_j).$$

      The quasimorphism inequality for $\mu_k$ yields the inequality

      \begin{equation*}
      \sum_{j=1}^N \mu_k(\ell_j f_j^{\pm 1} \ell_j^{-1}) - N \mathrm{def}(\mu_k)  \leq   \mu_k(h) \leq \sum_{j=1}^N \mu_k(\ell_j f_j^{\pm 1} \ell_j^{-1}) + N \mathrm{def}(\mu_k),
      \end{equation*}
where $\mathrm{def}(\mu_k)$ denotes the defect of $\mu_k$.  The quaimorphism inequality further implies that $$ - 4 \, \mathrm{def}(\mu_k) \leq \mu_k(\ell_j f_j^{\pm 1} \ell_j^{-1}) -\mu_k(f_j^{\pm 1}) \leq 4 \, \mathrm{def}(\mu_k).$$
      Combining the previous two inequalities we obtain
      \begin{equation}\label{eqn:quasi-ineq}
      \sum_{j=1}^N \mu_k( f_j^{\pm 1} ) - 5N \mathrm{def}(\mu_k)  \leq   \mu_k(h) \leq \sum_{j=1}^N \mu_k(f_j^{\pm 1} ) + 5N \mathrm{def}(\mu_k).
      \end{equation}

Note that as $k \to \infty$, 
\begin{itemize}
    \item[a.] $5N \mathrm{def}(\mu_k) \to 0$ because $N$ is independent of $k$ and $\mathrm{def}(\mu_k) \to 0$,
    \item[b.]  $\sum_{j=1}^N \mu_k( f_j^{\pm 1} ) \to \sum_{j=1}^N \hat{\mu}(f_j^{\pm 1}) $. 
\end{itemize}

Hence, we conclude from \eqref{eqn:quasi-ineq} that 
$$ \lim_{k \to \infty } \mu_k(h)= \sum_{j=1}^N \hat{\mu}(f_j^{\pm 1}).$$ \end{proof}
This completes the proof of Lemma \ref{lem:quasi-converge-hom}. 
\end{proof}

\subsection{Space filling links}

\emph{We give sufficient conditions on a sequence of Lagrangian links to define, via closed-open maps, a corresponding sequence of symmetric product idempotents satisfying the previously introduced conditions.}

Recall that a Lagrangian link $L \subset M$ of $k$ components is an embedding of $k$ pairwise disjoint Lagrangian submanifolds.  Such an $L$ determines a Lagrangian $\Sym(L) \subset \Sym^k(M)$ which lies away from the diagonal, hence has well-defined orbifold Lagrangian Floer cohomology groups by the theory established in \cite{Cho-Poddar}.  In this setting, there are Hamiltonian spectral invariants $c_k$ associated to any choice of bulk $\mathfrak{b}_k \in H^*(I\Sym^k(M);\Lambda_{>0})$ of positive valuation on the inertia orbifold of $\Sym^k(M)$. 
When $HF^{\mathfrak{b}}(\Sym(L),b)$ is well-defined (i.e.\ when $b$ is a Maurer-Cartan solution with respect to $L$ and $\mathfrak{b}$),  
there are also Lagrangian  spectral invariants $\ell_k$.
There are furthermore (length zero) closed-open and open-closed maps
\[
\mathcal{CO}^0: QH(\Sym^k(M),\mathfrak{b}_k) \to HF(\Sym(L_k),\mathfrak{b}_k,b); \quad \mathcal{OC}^0: HF(\Sym(L_k),\mathfrak{b}_k,b) \to QH(\Sym^k(M),\mathfrak{b}_k)
\]
and their cousins with (co)domain Hamiltonian rather than quantum cohomology, cf.\ Section \ref{sec:orb-Lag-Floer}. We will always focus attention on links in which each component of $L_k$ is a torus, so $\Sym(L_k)$ is also a torus. Its Floer theory is therefore governed by a disc potential function 
\[
W^{\mathfrak{b}_k} = W(\Sym(L_k);\mathfrak{b}_k): H^1(\Sym(L_k);U_{\Lambda}) \to \Lambda
\]
where $U_{\Lambda} = val^{-1}(1)$ denotes the unitary subgroup of the Novikov field.

\begin{prop}\label{prop:links}
Consider a Lagrangian torus link $L \subset M$ of $k$ components. Assume that there is a bulk class $\mathfrak{b}$ such that
\begin{itemize}
\item The function $W^{\mathfrak{b}}$ has a Morse critical point at $b \in H^1(\Sym^k(L);U_{\Lambda})$;
\end{itemize}
Let $Z$ be the determinant of the Hessian of $W^{\mathfrak{b}}$ at the critical point $b$.
Then 
\begin{enumerate}
\item the image of the length zero open-closed map $\mathcal{OC}^0_{\Sym^k(L)}$ has rank one so it contains an indecomposable idempotent $e \in QH^*(\Sym^k(M),\mathfrak{b})$ such that $e \cdot QH^*(\Sym^k(M),\mathfrak{b})=e \cdot \Lambda$;
\item $c(e, -)$ defines a quasi-morphism with defect at most $\val(Z)$;
\item for the unit class $e_L \in HF(\Sym^k(L),\mathfrak{b},b)$, we have $$\ell(e_L,H)+\val(Z) \ge c(e,H) \ge \ell(e_L,H)$$ for all $H$.  Here, $\ell(e_L, \cdot)$ stands for the Lagrangian spectral invariant associated to $e_L$. 
\end{enumerate}
Moreover, the maps induced on the universal cover $\widetilde{\Ham}(M, \omega)$ by the spectral invariants $\ell(e_L, \cdot), c(e, \cdot)$ are quasimorphisms, which coincide up to homogenization. 
\end{prop}

Recall that, for $\tilde{\psi} \in  \widetilde{\Ham}(M, \omega)$  are defined $\ell(e_L, \tilde{\psi}), c(e_k, \tilde{\psi})$ are defined to be $\ell(e_L, H), c(e_k, H)$, where $H$ is a mean-normalized such that $\tilde{\psi} = \tilde{\varphi}_H$.

The above result is proved in Section \ref{Sec:Lag Floer} and \ref{s:LagSpec}  (see Corollary \ref{c:OCCObound} and \ref{c:LagSpec}).
Readers should compare Proposition \ref{prop:links} with similar results proved in \cite[Chapter 21]{FOOOspectral} in the toric case.

\begin{corol}\label{c:links}
Consider a sequence $L_k \subset M$ of links of $k$ components, for $k=1,2,\ldots$. Assume that there are bulk classes $\mathfrak{b}_k$ such that
\begin{enumerate}
\item The function $W^{\mathfrak{b}_k}$ has an isolated non-degenerate critical point at $b_k \in H^1(\Sym^k(L_k);U_{\Lambda})$;
\item The valuation of the determinant of the Hessian of $W^{\mathfrak{b}_k}$ at the critical point $b_k$ grows sublinearly with $k$.
\end{enumerate}
Then, the image of the length zero open-closed map $\mathcal{OC}^0_{\Sym{L}_k}$ has rank one and contains  a unique idempotent $e_k \in QH^*(\Sym^k(M),\mathfrak{b}_k)$, and the  $\{e_k\}$ satisfy the hypotheses of Proposition \ref{prop:idempotent}.  Hence, the associated spectral invariants satisfy the Weyl law.
\end{corol}

\begin{remark}
The links that satisfy the hypotheses of Proposition \ref{prop:links} necessarily have to become {\it space filling} in the sense that the union $\cup_k L_k$ is a dense subset of $M$.  Indeed, otherwise one could violate the Weyl law by thinking about perturbations of a Hamiltonian supported in the complement of the union of the $L_k$ and the `Lagrangian control' property of spectral invariants.
\end{remark}

\begin{remark}\label{rem:quasi-morphisms}
    Interesting examples of Lagrangian links satisfying the hypothesis of Proposition \ref{prop:links} include links whose components are suitably chosen toric fibers in toric four manifolds, such as those appearing in \cite{MS19, Pol-Shel21}.  As a consequence of Proposition \ref{prop:links}, we see that the spectral invariants associated to such links are quasimorphisms.
    
     For a concrete example, one can take $(S^2\times S^2,\omega_{2B+C} + \omega_{2a})$, where the subscripts denote area, and the link comprising the product of $k$ circles in the first factor dividing the sphere into $k-1$ annuli of area $\frac{C}{k-1}$ and two discs of area $B$, with the equator in the second factor.  When $0<a < B- \frac{C}{k-1}$ the link is non-displaceable, and for suitable bulk its disc potential has a non-degenerate critical point; this is proven in \cite{MS19}, for $k=2$ and in \cite{Pol-Shel21} for $k>2$.

    This implies, in particular, that the invariants $c_{k,B}, \mu_{k,B}$, introduced in \cite{Pol-Shel21} are quasimorphisms. It can be shown that these differ from previously constructed\cite{Entov-Polterovich09} examples of quasimorphisms on $(S^2\times S^2,\omega_{2B+C} + \omega_{2a})$.
\end{remark}

\subsection{Weyl Lawlessness} 

\emph{We show that, if one does not non-trivially bulk-deform Floer cohomology, no sequence of  idempotents in the quantum cohomologies of symmetric products on $S^2$ satisfy the Weyl law. We are therefore obliged, in the sequel, to develop orbifold Hamiltonian Floer theory with bulk insertions.}

Proposition \ref{prop:links} reduces establishing the Weyl law to exhibiting a sequence of Lagrangian links of tori with certain properties. It is reasonable to ask why one requires Proposition \ref{prop:links} and not just Proposition \ref{prop:idempotent}. The point is that the existence of the idempotents required by the latter relies very delicately on choosing the correct bulk deformation classes $\mathfrak{b}_k$. This is illustrated by the following discussion in the simplest case.

\begin{prop}
If $M = \bP^1$ and each bulk class $\mathfrak{b}_k=0$, then there is no sequence of idempotents $e_k \in QH^*_{orb}(\Sym^k(\bP^1))$ which satisfy the two conditions of  Proposition \ref{prop:idempotent}.
\end{prop}

\begin{proof} 
Let $A_k = QH^*_{orb}(\Sym^k(\bP^1); \Lambda)$ where $\Lambda$ is the one-variable Novikov field with parameter $q$. We claim there is no sequence of  elements $e_k \in A_k$ for which 
\[
e_k \ast A_k = \Lambda\cdot e_k, \qquad \lim_k \, \val(e_k) / k = 0.
\]

An important feature when $\mathfrak{b}_k=0$ is that we can define a $\mathbb{Q}$-grading.
Indeed, Chen-Ruan prove that $H^*_{orb}(\Sym^k(\bP^1))$ is $\mathbb{Z}$-graded and the algebra structure respects the grading.
In other words, the moduli of constant spheres with $3$ (possibly orbifold) marked points ($2$ inputs and $1$ output) has vitual dimension given by the grading of the output minus the sum of the grading of the inputs.
When we consider the quantum product in the case $\mathfrak{b}_k=0$, we only need to consider the moduli of $J$-holomorphic spheres with  $3$ (possibly orbifold) marked points.
The virtual dimension of such $u$ in class $\beta$ is given by
\[
2\left(\langle c_1(TY),\beta\rangle + (\dim_{\bC}(Y)-3)(1-g) + 3 - \sum_{i=1}^3 a(u,r_i)\right)
\]
(see \cite{MSS}). 
It implies that the output has grading given by the sum of the gradings of the input minus $2\langle c_1(TY),\beta\rangle$.
Since $c_1(TY)$ is descended from $c_1(TX)$ and $\omega_Y$ is descended from $\omega_X$, the fact that $X$ is monotone implies that 
$c_1(TY)$ is proportional to $\omega_Y$ as well (with the same monotonicity constant).
In other words, we can define the grading of $q^{a\omega(\mathbb{P}^1)}$ to be $-4a$.
It gives a $\mathbb{Q}$-grading on $A_k$ and the product structure respects it.
As an idempotent, the grading of $e_k$ has to be $0$.

The condition $e_k \ast A_k = \Lambda\cdot e_k$ says that $e_k \ast h \in \Lambda e_k$ for every $h$, so $e_k$ is a simultaneous eigenvector for quantum multiplication by all elements $h\in A_k$. Now let $e_k^{(0)}$ denote the restriction of $e_k$ to the trivial sector  $A_k^{(\id)} \subset A_k$.  If we take $h$ in the trivial sector, then (since the bulk $b_k=0$) the output $e_k^{(0)} \ast h \in A_k^{(\id)}$ also belongs to the trivial sector, because pairs of pants contributing to the product have trivial monodromy at two boundaries and hence at the third (i.e. we are using the $\Gamma = \Sym_k$-grading). Thus, $e_k^{(0)}$ is a common eigenvector for quantum multiplication by elements $h\in A_k^{(\id)}$.

We completely understand $A_k^{(\id)}$ because the product structure only involves counting smooth $J$-holomorphic spheres.
Indeed, we have
\[
A_k^{(\id)} \cong QH^*((\bP^1)^k)^{\Sym_k}=(QH^*((\bP^1))^{\otimes k})^{\Sym_k}=((\Lambda[1_{\bP^1},H]/(H^2-q^{\omega(\bP^1)}))^{\otimes k})^{\Sym_k}
\]
 is exactly the ring of invariants in quantum cohomology of the product.
Therefore, $A_k^{(\id)}$ is semi-simple.
The indecomposable idempotents of $A_k^{(\id)}$ are given by tensor products of the indecomposable idempotents of 
$\Lambda[1_{\bP^1},H]/(H^2-q^{\omega(\bP^1)})$, namely, $\frac{1}{2}(1+q^{-\omega(\bP^1)/2}H)$ and $\frac{1}{2}(1-q^{-\omega(\bP^1)/2}H)$.
Note that, the valuation of any indecomposable idempotent of $A_k^{(\id)}$ is $-k\omega(\bP^1)/2$.

Since $e_k^{(0)}$ is a common eigenvector for quantum multiplication by elements $h\in A_k^{(\id)}$, $e_k^{(0)}$ is a multiple of an idempotent.
Using that the grading of $e_k^{(0)}$ is $0$, we know that $e_k^{(0)}$ is an idempotent times an element in $\mathbb{C}$.
Therefore, as long as $e_k^{(0)} \neq 0$, we have
\[
\val(e_k) \le \val(e_k^{(0)}) = -k\omega(\bP^1)/2
\]
As a result, the condition that $val(e_k)/k$ converges to $0$ cannot be satisfied.

It remains to see that $e_k^{(0)} \neq 0$ if $e_k \neq 0$. Let $e_k^{(g)}$ denote the restriction of $e_k$ to the sector labelled by $g\neq \id$ in the symmetric group, and $1_{g^{-1}}$ denote the unit in the sector labelled by $g^{-1}$.
Since $e_k \neq 0$, there is $g$ such that $e_k^{(g)} \neq 0$.
Note that $1_{g^{-1}} \ast e_k^{(g)}|_{A_k^{(0)}}  \neq 0$ because the classical contribution is non-zero and $1_{g^{-1}} \ast e_k^{(h)}|_{A_k^{(0)}}  = 0$ if $(h) \neq (g)$.
Therefore, we have
\[
0 \neq 1_{g^{-1}} \ast e_k^{(g)}|_{A_k^{(0)}}  = (1_{g^{-1}} \ast e_k)|_{A_k^{(0)}} 
\]
and the final term would vanish if $e_k^{(0)} = 0$.  
This completes the proof.
 \end{proof}

\begin{remark}
    The structure of the orbifold quantum cohomology $QH^*(\Sym^k(S^2))$ does not seem to have been determined in the literature. For $k=2$ one can check by hand that it is semisimple and it seems likely that holds in general (the $I$-function with a few insertions was computed in \cite{Silversmith} subject to some conjectural complicated combinatorial identities). 
\end{remark}

The preceding result indicates that it is not sufficient to work with bulk parameter zero, and `continuity results' in Floer cohomology then suggest that one should not expect the required idempotents to exist in a Zariski open subset of possible bulk classes.  We will  infer the existence of suitable idempotents from appropriate Lagrangian links, via Proposition \ref{prop:links}; in this case, the required bulk class is `picked out' (by an order-by-order induction in the adic filtration on the Novikov ring) by the requirement that Floer cohomology of the link is non-vanishing.

\subsection{The Weyl law on $S^2$}\label{sec:Weyl S2}

\emph{We show that the sequence of links given by taking more and more parallel circles on $S^2$ satisfy all our required conditions. This yields a new proof of the Weyl law, and hence the smooth closing lemma for area-preserving diffeomorphisms, in this case.}

Let $L $ be a Lagrangian link on $S^2$ consisting of $k$ parallel circles which is $\eta$-monotone in the sense of \cite{CGHMSS} for some $\eta > 0$.

Thus, we require that the area of the annuli components in $S^2 \setminus L$ are the same, denoted by $A_k$, and the area of the two disc components are the same, denoted by $B_k$.
We further require that $A_k<B_k$. Let $[T]$ be the cohomological unit of the transposition sector in $I\Sym^k(S^2)$ and $\mathfrak{b}:=c[T] \in H(IY)=H_{orb}(Y)$.

\begin{lemma}[cf. \cite{MS19, Pol-Shel21}]\label{l:parallelLink}
There is a $c \in \Lambda_{>0}$ with $\val(c)=\frac{B_k-A_k}{2}$ such that $\Sym^k(L)$ is weakly unobstructed and the bulk-deformed potential function $W^{\mathfrak{b}}$ of $\Sym^k(L)$ has a non-degenerate critical point.
\end{lemma}

\begin{proof}
Let $u$ be a representable holomorphic orbifold disc in $\Sym^k(S^2)$ with boundary on $\Sym^k(L)$ such that all the orbifold marked points are mapped to the transposition sector. Let $h$ be the number of orbifold marked points.
Let $v:\Sigma \to S^2$ be the associated map under the tautological correspondence.
The virtual dimension of $u$ is given by $(k-3)+ \mu(v)+h$, where $\mu(v)$ is the Maslov index of $v$ (see \cite{MS19}).
It can contribute to the obstruction only if $(k-3)+ \mu(v)+h \le k-2$, equivalently $\mu(v)+h \le 1$.
Since the Maslov index of the disc components and annuli components in $S^2 \setminus L$ are $2$ and $0$, respectively, and the positivity of intersection guarantees that the relative homology class $[v]$ is a non-negative linear combination of  these components, the only case where $\mu(v) \le 1$ is $[v]$ being a linear combination of the annular components.
It implies that, if $v$ is not a constant map, then at least one of the components of its domain is not a disc.
However, since $\Sigma$ comes with a $k$-fold branch cover to the disc by the tautological correspondence, having one non-disc component implies that $h$ is at least $2$.
Therefore, no $u$ can contribute to the obstruction and $\Sym^k(L)$ is weakly unobstructed.

To compute the potential, we only need to look at those $u$ such that $(k-3)+ \mu(v)+h=k-1$, equivalently, $\mu(v)=2$ and $h=0$, or $\mu(v)=0$ and $h=2$.
The case that $\mu(v)=2$ and $h=0$ means that the $k$-fold covering from $\Sigma$ to the disc is unramified and hence $\Sigma$ is a disjoint union of $k$ discs.
 The lowest order contributions are given by the two disc components, contributing  the terms $T^{B_k}z_1$ and $T^{B_k}\frac{1}{z_k}$ respectively.

The case that $\mu(v)=0$ and $h=2$ means that the $k$-fold branched covering from $\Sigma$ to the disc has exactly two branched points and the image of $v$ does not intersect the two disc components of $S^2 \setminus L$.
  The lowest order contribution are given by each annular component, contributing  the terms $c^2T^{A_k}(\frac{1}{z_j}+z_{j+1})$ for $j=1,\dots,k-1$.

It is routine to check that we can find $c$ with $\val(c)=\frac{B_k-A_k}{2}$ such that $W^{\mathfrak{b}}$ has a non-degenerate critical point (cf. \cite{MS19}).
\end{proof}

\begin{corol}\label{cor:Weyl-S2}
    The Weyl law for $S^2$ holds. 
\end{corol}

\begin{proof}
We now consider Lemma \ref{l:parallelLink} for the  sequence of links $L_k$ comprising $k$ parallel circles, with $k\to \infty$.  
    By construction, the determinant of the Hessian of the disc potential function has valuation $kB_k$. This does grow sublinearly with $k$, since $B_k \to 0$ as $k\to\infty$.  Thus, the links satisfy the conclusions of Proposition \ref{prop:links}. 
    \end{proof}

This completes the proof of the Weyl law and hence the smooth closing lemma on $S^2$. 

\begin{remark}
It is hard to find Lagrangian links with more than one component which  satisfy Proposition \ref{prop:links} in higher dimensions.
For example, suppose that $M=(S^2 \times S^2, \omega+t\omega)$ and $L_k$ is a disjoint union of product Lagrangians $\{C_{1j} \times C_{2j}: j=1,\dots,k\}$. Say $C_{11}$ bounds the smallest area disc (with area $a$) among all $C_{ij}$.
For the potential function of $\Sym(L_k)$ to have a critical point in $ H^1(\Sym^k(L);U_{\Lambda})$, we need the truncation of it by order greater than $T^a$ to have a critical point. However, for  Maslov index reasons, only the disc(s) bounded by $C_{11}$ can contribute to the potential with order at most $T^a$. Therefore, for it to have a critical point, we need $C_{11}$ to bound two discs of equal area. In particular $C_{11}$ has to be Hamiltonian isotopic to the equator in the $S^2$ factor with smaller area.
Therefore, even if $L_k$ satisfies Proposition \ref{prop:links} (these kinds of links exist by \cite{MS19, Pol-Shel21}), we cannot find a sequence of them whose union is dense in $M$.
\end{remark}

\subsection{Periodic orbits meeting displaceable open sets}

\emph{In this section we explain how to construct periodic orbits near certain displaceable open sets, as a further application of orbifold Hamiltonian spectral invariants.}

Consider a Lagrangian link $L \subset M$ of $k$ components satisfying the hypotheses of Proposition \ref{prop:links}, meaning we assume the existence of a bulk class $\mathfrak{b}$ such that the function $W^{\mathfrak{b}}$ has an isolated non-degenerate critical point at $b \in H^1(\Sym(L);U_{\Lambda})$.  As in Proposition \ref{prop:links}, let $Z$ be the determinant of the Hessian of $W^{\mathfrak{b}}$ at the critical point $b$.

Let $T_1, \ldots, T_k$ denote the components of the link $L$.
\begin{theo}\label{t:links and orbits}
Let  $G: [0,1] \times M \rightarrow \R$ be a smooth Hamiltonian such that $G|_{T_1} > \val(Z)$ and $G=0$ on a neighborhood of the link components $T_2, \ldots, T_k$.  Then, for any Hamiltonian diffeomorphism $\psi$, there exists $s\in [0,1]$ such that the composition $\psi \varphi^s_{G}$  has a periodic point of period at most $k$ passing through the support of $G$.
\end{theo}
We remark here on some consequences of this result before proceeding to prove it.
\begin{remark}\label{rem:links and orbits}
    \begin{itemize}
         \item Interesting examples where one can apply the above are Lagrangian links consisting of suitable toric fibers in toric four manifolds, such as those appearing in \cite{MS19, Pol-Shel21}. (For a concrete example, one can take $(S^2\times S^2,\omega_{2B+C} + \omega_{2a})$, where the subscripts denote area, and the link comprising the product of two circles in the first factor dividing the sphere into an annulus of area $C$ and two discs of area $B$, with the equator in the second factor.  When $0<a < B-C$ the link is non-displaceable, and for suitable bulk its disc potential has a non-degenerate critical point.)  In these cases, the individual components of the links are displaceable and, as far as we can see, a periodic point created via a perturbation supported near an individual link component is not detectable via ordinary Floer homology.
         \item One can deduce the smooth closing lemma for Hamiltonian diffeomorphisms of $S^2$, by applying the above to the Lagrangian links considered in Section \ref{sec:Weyl S2}, which consist of parallel circles.  This yields a proof of the closing lemma which detects variation of spectral invariants through a mechanism other than the Weyl law.   Edtmair and Hutchings \cite{EH21} offer a different proof of the closing lemma on the $2$-torus which does not explicitly use the Weyl law.
    \end{itemize}
\end{remark}

\begin{proof}[Proof of Theorem \ref{t:links and orbits}]
Fix a Hamiltonian diffeomorphism $\psi$ of $M$ and let $H$ be a generating Hamiltonian for $\psi$, i.e.\ $\psi = \varphi_H$.  Let $e_L \in HF(\Sym(L_k),\mathfrak{b}_k,b)$  and $e_k \in QH^*(\Sym^k(M), \mathfrak{b})$ be as in the conclusion of Proposition \ref{prop:links}.  Then, according to the proposition, we have the following inequalities
\begin{align}
\ell(e_L,H) & \le c(e_k,H) \le  \ell(e_L,H) + \val(Z), \label{ineq1}  \\
\ell(e_L, H \# G) &\le c(e_k, H \# G) \le  \ell(e_L, H \# G) + \val(Z). \label{ineq2}
\end{align}

\begin{claim}\label{cl:shift-spec}
$\ell(e_L, H) + \val(Z) < \ell(e_L, H \# G).$   
\end{claim} 
Before proving the above claim, we explain how it implies the proposition.
From Claim \ref{cl:shift-spec}, combined with inequalities \eqref{ineq1} \& \eqref{ineq2}, we conclude that 
\begin{equation}\label{eqn:not-equal}
    c(e_k, H) < c(e_k,  H \# G).
\end{equation}  

Suppose that for all $s \in [0,1]$, the Hamiltonian diffeomorphism $\psi \varphi^s_{G}$ has no periodic point of period at most $k$ meeting the support of $G$. Then, one can argue as in the proof of Lemma \ref{l:closing} that the function $s \mapsto c(e_k,  H \# sG )$ is constant. Indeed, the assumption on non-existence of periodic points with period at most $k$ in the support of $G$ implies that $\Spec_k(H \# sG ) = \Spec_k(H)$, for $s\in [0,1]$.  This in turn implies that the continuous function $s \mapsto c(e_k, H \# sG)$ takes values in the measure zero set $\Spec_k(H)$ and so it must be constant.  Hence, we must have $c(e_k, H) = c(e_k,  H  \# G)$ which clearly contradicts \eqref{eqn:not-equal}.

It remains to prove Claim \ref{cl:shift-spec}.
\begin{proof}[Proof of Claim \ref{cl:shift-spec}]
Let $F$ be a Hamiltonian such that  $F \leq G$, and $F|_{T_1} = C > \val(Z) $, where $C$ is some constant, and $F=0$ on a neighborhood of the link components $T_2, \ldots, T_k$. Then, $H \# F \leq H\# G$ and so by Monotonicity we have 
$ \ell(e_L, H \# F) \leq \ell(e_L, H \# G) $.  Thus, it suffices to prove $\ell(e_L, H) + \val(Z) < \ell(e_L, H \# F)$.  As we explain now, we in fact have the identity  
\begin{equation}\label{eqn:Lag-spec1}
    \ell(e_L, H \# F) = \ell(e_L, H ) + C,
\end{equation}
which clearly implies $\ell(e_L, H) + \val(Z) \le \ell(e_L, H \# F)$.  The above identity holds because, for each $s \in [0,1]$ we have 
\begin{equation} \label{eqn:Lag-spec2}
    \Spec_k(H \# sF;L) = \Spec_k(H;L) + s C,
\end{equation}
where $\Spec_k( \cdot ;L)$ denotes the Lagrangian action spectrum.  Indeed, \eqref{eqn:Lag-spec2} implies that $\ell(e_L, H \# sF) - sC \in \Spec_k(H;L)$ which forces the function $s \mapsto \ell(e_L, H \# sF) - sC$ to be constant because $\Spec_k(H;L)$ is of measure zero. Setting $s=1$, then yields \eqref{eqn:Lag-spec1}.  

Finally, here is why \eqref{eqn:Lag-spec2} holds. Note that 
the  Hamiltonian flows of $\Sym^k(H \# sF)$ and $\Sym^k(H)$ coincide on $\Sym^k(L)$; this is a consequence of the fact that $F = C$ on a neighborhood of $T_1$ and that $F$ is supported away from the other components of the link.  Hence, there is a natural  1-1 correspondence between the (capped) Hamiltonian chords, which begin and end on $\Sym^k(L)$, for the two (symmetric product) Hamiltonians $\Sym^k(H \# sF)$ and $\Sym^k(H)$.  Moreover, under this correspondence actions of capped chords differ by the amount $sC$.  This completes the proof of Claim \ref{cl:shift-spec}.
\end{proof} 
This completes the argument.
\end{proof}

\section{Orbifold Lagrangian Floer theory}\label{sec:orb-Lag-Floer}

\subsection{Recap of the theory}

Lagrangian Floer theory for a Lagrangian submanifold disjoint from the orbifold locus has been constructed in \cite{Cho-Poddar}, using the technology of Kuranishi atlases in the sense of Fukaya-Oh-Ohta-Ono \cite{FOOOtoric, FOOOspectral}.  We briefly recall the output of that theory.

Let $\frak{b} \in H^*(IY;\Lambda_{>0})$ be a bulk deformation class; we will always take this to be a multiple of the fundamental class of some non-trivial twisted sector.  The quantum cohomology $QH^*_{orb}(Y;\frak{b})$ has underlying vector space $H^*_{orb}(Y)$, and a ring structure which deforms the Chen-Ruan product.  

Let $L\subset Y^{reg}$ be a compact Lagrangian submanifold which is contained in the regular locus of $Y$. It has a bulk-deformed orbifold Lagrangian Floer cohomology $HF(L,L;\frak{b})$, see  \cite{Cho-Poddar} and \cite[Section 2.2]{MS19} for a brief summary of the theory.   The output can be summarised in our language as:

\begin{prop}[\cite{Cho-Poddar}] The Floer ordered marked flow category associated to the pair $L$ and $\frak{b}$ admits an enrichment in Kuranishi atlases / diagrams. \end{prop}

\begin{remark} \label{rmk:lagrangian via global charts}
Given our construction of orbifold Hamiltonian Floer theory in \cite{MSS}, it is natural to ask for a construction of $HF(L,L;\frak{b})$ using global charts rather than Kuranishi atlases. This can likely be achieved by combining the theory from \cite{Hirschi-Hugtenburg}  with the technology of admissible covers, but we have not attempted to carry this out in detail.
\end{remark}

\subsection{Formal structure of  orbifold Lagrangian Floer theory\label{Sec:Lag Floer}}

\emph{We summarise some of the known and expected features of orbifold Lagrangian Floer cohomology which replicate well-known structure in the usual case.}

 Most of the results below follow directly from the work of \cite{Cho-Poddar} and \cite{FOOOspectral}, but some features of open-closed and closed-open maps in this setting have not been written explicitly in the literature. The proofs would be minor variations on the corresponding results in Hamiltonian Floer theory established in \cite{MSS}, but using the Lagrangian Floer global charts from \cite{Hirschi-Hugtenburg}. We axiomatise these remaining pieces to keep the paper of manageable length.

 Let $L$ be an oriented, spin, closed Lagrangian submanifold in $Y^{reg}$.

\begin{theo}[$A_{\infty}$ structure]\label{a:infty}
The complex $CF^*(L,L;\frak{b},b)$ has a unital curved $A_{\infty}$ structure $\{\tilde{m}^{\frak{b},b}\}$ defined by moduli of orbifold holomorphic discs where the unit is given by the constant function $1$.
\end{theo}

\begin{remark}
The Floer complex $CF^*(L,L;\frak{b},b)$ is canonically $\mathbb{Z}/2\mathbb{Z}$-graded because $L$ lies in the trivial sector of $IY$ and $\frak{b} \in QH^{even}_{orb}(Y)$. The signs of the $A_{\infty}$ equations follow  the Koszul sign rule with respect to this  $\mathbb{Z}/2\mathbb{Z}$-grading.
\end{remark}

We denote the unit by $e_L$ and  the space of weak bounding cochains for $\{\tilde{m}^{\frak{b}}\}$ by
$\widehat{\cM}_{\mathrm{weak}}(L, \tilde{m}^{\frak{b}})$.

Suppose now that $L$ is a Lagrangian torus such that $H^1(L,\Lambda_{\ge0}) / H^1(L, 2i\pi\Z) \subset \widehat{\cM}_{\mathrm{weak}}(L, \tilde{m}^{\frak{b}})$.
Let  $W^{\frak{b}}: H^1(L;\Lambda_{\ge0}) \to \Lambda_{>0}$ be defined by
\[
W^{\frak{b}}(b)e_L:=\sum_{k=0}^{\infty} \tilde{m}_k^{\mathfrak{b},b_0}(b_+,\dots,b_+)
\]
for $b=b_0+b_+$, $b_0 \in H^1(L,\mathbb{C})$ and $b_+ \in H^1(L,\Lambda_{>0})$.
We call $W^{\frak{b}}$ the curvature potential.

On the other hand, we define the disk potential
$W^{\frak{b}}_{disc}: H^1(L;\Lambda_{\ge0}) \to \Lambda_{>0}$
by
\[
W^{\frak{b}}_{disc}(b)e_L:=
\sum_{\beta} T^{\omega_Y(\beta)}\exp(\mathfrak{b} \cdot \beta)\exp(b(\partial \beta))\tilde{m}_{0,\beta}(1)
\]

\begin{theo}[curvature equals to the disk potential]\label{a:divisor}
The divisor axiom holds for the $A_{\infty}$ structure $\{\tilde{m}^{\frak{b},b}\}$ so that $W^{\frak{b}}(b)=W^{\frak{b}}_{disc}(b)$.
\end{theo}

As a consequence of Theorem \ref{a:divisor}, we have the following.
 
 \begin{lemma}[\cite{Cho-Poddar},
\cite{Cho-Clifford} Theorem 5.6, \cite{FOOOtoric} Theorem 3.6.2, Proposition 3.7.1, \cite{SheridanFano} Proposition 4.3]
Let $L$ be as above. If $b\in H^1(L;\Lambda_{\ge0})$ is a critical point of $W^{\frak{b}}_{disc}$, then it defines a local system over $L$ for which $HF(L,L;\frak{b},b)=H^*(L;\Lambda)$. If $b$ is a non-degenerate critical point, then $HF(L,L;\frak{b},b)$ is multiplicatively isomorphic to the Clifford algebra $Cl_{\dim(L)}$ over $\Lambda$.
\end{lemma}

From now on, we assume that $b$ is a non-degenerate critical point of $W^{\frak{b}}_{disc}$ so $(HF(L,L;\frak{b},b),\tilde{m}_2^{\mathfrak{b},b})$ is a Clifford algebra. We recall some consequences.

\begin{enumerate}
\item The $A_{\infty}$ structure on $HF(L,L;\frak{b},b)$ is formal so it is quasi-isomorphic to $Cl_{\dim(L)}$. \cite[Corollary 6.4]{SheridanFano}
\item $HH_*(Cl_{\dim(L)})=\Lambda$ and $HH^*(Cl_{\dim(L)})=\Lambda$ 
\item as a $\mathbb{Z}/2\mathbb{Z}$-graded vector space, we have
$Cl_{2k}=\End((\Lambda[0] \oplus \Lambda[1])^{\otimes k})$
and $Cl_{2k+1}=\End((\Lambda[0] \oplus \Lambda[1])^{\otimes k}) \otimes Cl_1$
\end{enumerate}

Whilst the following result is not stated explicitly in \cite{Cho-Poddar}, given their framework it can be deduced as in the smooth case in \cite{FOOOtoric}.

\begin{theo}[cyclic $A_{\infty}$ structure]\label{a:cyclic}
 $HF(L,L;\frak{b},b)$ has a cyclic unital $A_{\infty}$ structure $(\{m^{\frak{b},b}\}, 
\langle \cdot, \cdot \rangle)$ such that
\begin{enumerate}
\item $\langle \cdot, \cdot \rangle$ is given by the Poinc\'are pairing,
\item there is an $A_{\infty}$ quasi-isomorphism $F:(HF(L,L;\frak{b},b),\tilde{m}_k^{\mathfrak{b},b}) \to (HF(L,L;\frak{b},b),m_k^{\mathfrak{b},b})$ such that 
$F=id$ mod $\Lambda_{>0}$.
\end{enumerate}
\end{theo}

\begin{remark}
Since $HF(L,L;\frak{b},b)$ is quasi-isomorphic to a Clifford algebra, it is trivial to define a cyclic unital $A_{\infty}$ structure on it. The content of Theorem \ref{a:cyclic} is that the pairing comes from the Poinc\'are pairing on $L$.
\end{remark}

\begin{remark}
When $Y$ is a smooth symplectic manifold, Theorem \ref{a:cyclic} is verified in \cite[Theorem 3.2.9, Corollary 3.2.22]{FOOOtoric}.
It is unclear if one can define a cyclic unital $A_{\infty}$ structure such that the curvature potential agrees with the disk potential (cf. \cite[Remark 1.3.26, Remark 3.2.31]{FOOOtoric}).
\end{remark}

With the unital cyclic structure, we can define the trace of $HF(L,L;\frak{b},b)$ as follows.
Let $\{e_I: I \subset \{1,\dots,n\}\}$ be a basis of $H^*(L)$.
Let $g_{IJ}=\langle e_I, e_J \rangle$ and $g^{IJ}$ be the inverse matrix.
Let $vol \in HF(L)$ be the volume class (not the unit).
The trace is given by
\[
Z(L,\frak{b},b):=\sum_{I,J} (-1)^* g^{IJ} \langle m_2^{\mathfrak{b},b}(e_I,vol), m_2^{\mathfrak{b},b}(e_J,vol) \rangle
\]

\begin{remark}
Fukaya-Oh-Ohta-Ono give a general definition of the trace in \cite[Definition 1.3.22]{FOOOtoric} for unital Frobenius algebras. For our case where the pairing is given by the Poinc\'are pairing, it specializes to the definition above (cf. \cite[Theorem 3.4.1, 3.5.2]{FOOOtoric} , see also \cite[Equation (2.11)]{Sanda}) 
\end{remark}

\cite[Proposition 3.7.4]{FOOOtoric} explains how to compute the leading order term of $Z(L,\frak{b},b)$.

\begin{prop}\label{p:Z}
Given Theorems \ref{a:infty}, \ref{a:divisor}, \ref{a:cyclic}, we have
\[
Z:=Z(L,\frak{b},b)=\mathrm{det}(\Hess(W^{\frak{b}}_{disc})|_b) \text{ mod }T^{\val(Z)}\Lambda_{>0}
\]
In particular, since we have assumed that $b$ is a Morse critical point, we have $Z \neq 0$.
\end{prop}

On top of properties of Lagrangian Floer theory itself, we also need some compatibility with quantum cohomology.

\begin{assumption}\label{a:OCCO}
    There are `closed-open' (algebra) and `open-closed' ($QH^*_{orb}(Y;\frak{b})$-module) maps 
\[
\mathcal{CO}: QH^*_{orb}(Y;\frak{b}) \to HH^*(CF(L,L;\frak{b},b)) \qquad \mathrm{and} \qquad \mathcal{OC}: HH_*(CF(L,L;\frak{b},b)) \to QH^*_{orb}(Y;\frak{b}).
\]
Denote the `zeroth order' terms by 
\[
\mathcal{CO}_L^0: QH^*_{orb}(Y;\frak{b}) \to HF^*(L,L;\frak{b},b) \qquad \mathrm{and} \qquad \mathcal{OC}^0_L: HF^*(L,L;\frak{b},b) \to QH^*_{orb}(Y;\frak{b}).
\]
We assume that
    \begin{enumerate}
\item $\langle \mathcal{OC}^0(a),y \rangle_Y=\langle a,\mathcal{CO}^0(y) \rangle_L$ for all $a$ and $y$
\item $\langle \mathcal{OC}^0(a),\mathcal{OC}^0(a') \rangle_Y=\sum_{I,J} (-1)^* g^{IJ} \langle m_2^{\mathfrak{b},b}(e_I,a), m_2^{\mathfrak{b},b}(e_J,a') \rangle$ for all $a,a'$
\item $\mathcal{CO}_L^0$ and $\mathcal{OC}^0_L$ have non-negative valuation
    \end{enumerate}
\end{assumption}

\begin{remark}
The first equation simply comes from changing inputs to outputs in the consideration of moduli. The second equation uses the degeneration of annuli and is called the Cardy relation in the literature.
These equations are proved in \cite[Theorem 3.3.8 and Theorem 3.4.1]{FOOOtoric}.
\end{remark}

\begin{remark}
Assumption \ref{a:OCCO} should hold for any oriented, spin, closed Lagrangian submanifold $L$ with a bounding cochain $b$ solving the Maurer-Cartan equation. But we do not need this level of generality.
\end{remark}

For $(L,b)$ as above, Assumption \ref{a:OCCO} has the following consequences
\begin{enumerate}
\item for any $y$, $\mathcal{CO}^0_L(y)=c_y e_L$ for some $c_y \in \Lambda$ (because $HH^*(CF(L,L;\frak{b},b))$ is rank $1$ generated by the unit)
\item if $a \in \oplus_{d<n} H^d(L)$, then $\langle \mathcal{OC}^0(a),y \rangle_Y=\langle a,c_y e_L\rangle_L=0$ for any $y$
\item $\langle \mathcal{OC}^0(vol),\mathcal{OC}^0(vol) \rangle_Y=Z(L,\mathfrak{b},b) \neq 0$ (see Proposition \ref{p:Z})
\item $CO \circ OC:HH_*(CF(L,L;\frak{b},b)) 
\to HH^*(CF(L,L;\frak{b},b))$ is an isomorphism (because both are of rank $1$ and $Z \neq 0$ implies the non-triviality of the map)
\item there is a module property (see  \cite[Equation (4.2)]{Sanda}): for any $a \in HF(L,L;\frak{b}) $ and $y \in QH^*_{orb}(Y;\frak{b}) $, we have 
\[
\mathcal{OC}^0_L(\mathcal{CO}^0_L(y)a)=y\mathcal{OC}^0_L(a). 
\]
\end{enumerate}

\begin{corol}[cf. \cite{Sanda} Proposition 5.3]\label{c:OCCObound}
Suppose $L$ is a Lagrangian torus such that all the assumptions above are satisfied.
Then the followings are true:
\begin{enumerate}
\item The image of $\mathcal{OC}^0_L$ is a field summand of $QH^*_{orb}(Y;\frak{b})$. We denote the idempotent in the field summand by $e_{Y,L}$.
\item The valuation of $e_{Y,L}$, is at least $-\val(Z)$ so the spectral invariant of $e_{Y,L}$ defines a quasi-morphism with defect at most $\val(Z)=\val(\mathrm{det}(\Hess(W^{\frak{b}}_{disc})|_b))$.
\end{enumerate}
\end{corol}

\begin{proof}

We didn't assume that $\mathcal{OC}^0_L(e_L)=PD([L])$ (Assumption 4.3 (4) of \cite{Sanda}) so the first statement does not strictly follow from \cite[Proposition 5.3]{Sanda}.
However, it is easy to check that we do not need this assumption.
For the sake of the reader, we reproduce the argument here.
On the contrary, even though the second statement is an easy consequence, the explicit bound on the defect, to the authors' knowledge, does not appear in the literature, so we explain how it is obtained.

Let $I$ be the image of $\mathcal{OC}^0_L$. Since $CO \circ OC$ is an isomorphism, we know that $I$ has rank $1$.
By the module property, we know that $I$ is an ideal of $QH^*_{orb}(Y;\frak{b})$.
Let $y=\mathcal{OC}^0(vol)$.
We know that $\mathcal{CO}^0_L(y)=c_y e_L$ for some $c_y \in \Lambda$.
Together with 
\[
\langle \mathcal{CO}^0_L(\mathcal{OC}^0(vol)),vol \rangle_Y=\langle \mathcal{OC}^0(vol),\mathcal{OC}^0(vol) \rangle_Y=Z,
\]
we know that $c_y=Z$.
Let $e_{Y,L}:=\frac{\mathcal{OC}^0_L(vol)}{Z}=\frac{y}{Z}$ so that $\mathcal{CO}^0_L(e_{Y,L})=e_L$.
Then we have 
\[
e_{Y,L} \cdot e_{Y,L}=e_{Y,L} \cdot \frac{\mathcal{OC}^0_L(vol)}{Z}=\frac{1}{Z} \mathcal{OC}^0_L(\mathcal{CO}^0_L(e_{Y,L}) \cdot vol)=\frac{\mathcal{OC}^0_L(vol)}{Z}=e_{Y,L}
\]
Let the unit of $QH^*_{orb}(Y;\frak{b})$ be $e_Y$.
To show that we have a direct sum decomposition $QH^*_{orb}(Y;\frak{b})=I \oplus (e_Y-e_{Y,L})QH^*_{orb}(Y;\frak{b})$, we check that (by the module property again)
\[e_{Y,L} \cdot (e_Y-e_{Y,L})=\frac{1}{Z}(\mathcal{OC}^0_L(vol \cdot \mathcal{CO}^0_L(e_Y-e_{Y,L})))=0\]
because $\mathcal{CO}^0_L(e_Y-e_{Y,L})=e_L-e_L=0$.

Therefore, $I$ is a field summand.
Moreover, the idempotent is given by $e_{Y,L}:=\frac{\mathcal{OC}^0_L(vol)}{Z}$.
Since $\mathcal{OC}^0_L$ increases valuation, the valuation of $e_{Y,L}$ is at least $-\val(Z)$.
The proof that the spectral invariant of $e_{Y,L}$ defines a quasi-morphism with defect at most $\val(Z)=\val(\mathrm{det}(\Hess(W^{\frak{b}}_{disc})|_b))$ follows from \cite{Entov-Polterovich}.

\end{proof}

\subsection{Orbifold Lagrangian spectral invariants}\label{s:LagSpec}

\emph{We axiomatize the properties of orbifold Lagrangian Floer spectral invariants.}

The results of this section are not required for our template proof of the Weyl law, or the application to the closing lemma on $S^2$. Nonetheless, the results are relevant to the general question raised in the introduction of creating new periodic orbits meeting a set $U$ by a perturbation of the Hamiltonian supported on $U$. Since this is slightly orthogonal to our main aim, we will present fewer details in this section.

In parallel to the previous section, we list the Floer theoretic  properties which are needed to define Lagrangian spectral invariants and establish their properties.  Some of these are encompassed by the Hamiltonian invariance of orbifold Lagrangian Floer theory, as asserted in \cite{Cho-Poddar}, whilst others require compatibility with open-closed morphisms. Note in particular that in establishing Hamiltonian invariance, \cite{Cho-Poddar} construct Kuranishi atlases underlying the moduli spaces which define $CF^*(L,K,\mathfrak{b},b)$ for a pair of distinct Lagrangians $L,K$ lying in the regular locus of an orbifold. We will be exclusively concerned with the case in which $K$ is the Hamiltonian image $\phi_H^1(L)$ of $L$ (noting that a smooth Hamiltonian on the orbifold will preserve the regular stratum), in which case we will write the corresponding complex as $CF(L,H,\mathfrak{b},b)$.

Thus, let $(L,b)$ be as above.
Let $H$ be a Hamiltonian function such that $L \pitchfork \phi_H^1(L)$ and $CF(L,H,\mathfrak{b},b)$ be the vector space generated by time-$1$ Hamiltonian chord of $X_H$ from $L$ to $L$.

\begin{theo}
Orbifold Lagrangian Floer cohomology has the following properties:
\begin{enumerate}
    \item  $CF(L,H,\mathfrak{b},b)$ admits a differential (with non-negative valuation) making it into a chain complex.
    \item  There are chain maps $\Phi_{0,H}: CF(L,L,\mathfrak{b},b) \to CF(L,H,\mathfrak{b},b)$  and $\Phi_{H,0}: CF(L,H,\mathfrak{b},b) \to CF(L,L,\mathfrak{b},b)$ which are quasi-inverse to each other. The valuation of them decreases by at most the Hofer norm of $H$.
    \item If $H'$ is another Hamiltonian such that $L \pitchfork \phi_{H'}^1(L)$, then there are chain maps $\Phi_{H,H'}: CF(L,H,\mathfrak{b},b) \to CF(L,H',\mathfrak{b},b)$  and $\Phi_{H',H}: CF(L,H',\mathfrak{b},b) \to CF(L,H,\mathfrak{b},b)$ which are quasi-inverse to each other. The valuation of them decreases by at most the Hofer norm of $H-H'$.
    \item If $L \pitchfork \phi_{H\#H'}^1(L)$, then there is a valuation non-decreasing triangle product map
    $\Phi_{H,H', H\#H'}: CF(L,H,\mathfrak{b},b) \otimes CF(L,H',\mathfrak{b},b) \to CF(L,H \# H',\mathfrak{b},b)$ which can be descended to a bilinear map on the cohomology
    \item The maps $\Phi_{0,H}, \Phi_{H,0}, \Phi_{H,H'}, \Phi_{H',H}$ for all $H, H'$ are all compatible with each other and form a direct system; they are also compatible with $\Phi_{H,H', H\#H'}$ and the product structure on $HF(L,L,\mathfrak{b},b)$.
    
\end{enumerate}
\end{theo}

From these assumptions, we can define the Lagrangian spectral invariant as usual.

\begin{definition}
Let $e \in HF(L,L,\mathfrak{b},b) $, $H \in C^{\infty}([0,1] \times Y)$ be a non-degenerate Hamiltonian and $\mathfrak{b}$ as above.
The bulk-deformed spectral invariant is
\[
\ell(e,H) := \ell^{\mathfrak{b}}_{(L,b)}(e,H)= \inf \{a \in \mathbb{R}| \Phi_{0,H}(e) \in \im (HF(L,H,\mathfrak{b},b)^{<a} \to HF(L,H,\mathfrak{b},b))\}
\]

\end{definition}

It satisfies spectrality, Hofer-Lipschitz continuity, Lagrangian control, homotopy invariance and the triangle inequality.

\begin{assumption}
There are `closed-open' and `open-closed' maps (both valuation non-decreasing)
\[
\mathcal{CO}_{(L,H)}^0: HF^*_{orb}(H;\frak{b}) \to HF^*(L,H;\frak{b},b)  \qquad \mathrm{and} \qquad \mathcal{OC}^0_{(L,H)}: HF^*(L,H;\frak{b},b) \to HF^*_{orb}(H;\frak{b})
\]
such that $\mathcal{CO}_{(L,H)}^0 \circ PSS_H=\Phi_{0,H} \circ \mathcal{CO}_{L}^0$ and $\mathcal{OC}_{(L,H)}^0 \circ \Phi_{0,H}= PSS_H\circ \mathcal{CO}_{L}^0$.
\end{assumption}

\begin{corol}\label{c:LagSpec}
Under the hypothesis of Corollary \ref{c:OCCObound} and for the identity class $e_L \in HF(L,L;\frak{b},b)$, we have $\ell(e_L/Z,H) \ge c(e_{Y,L},H) \ge \ell(e_L,H)$ for all Hamiltonians $H$. So in particular, $\ell(e_L,-)$ also defines a quasi-morphism with defect at most $\val(Z)$ and its homogenization is the same as that of $c(e_{Y,L},-)$.
\end{corol}

\begin{proof}
We follow the same notations in  Corollary \ref{c:OCCObound}.
The fact that $e_{Y,L}=\frac{\mathcal{OC}^0_L(vol)}{Z}$ implies that $\ell(vol/Z,H) \ge c(e_{Y,L},H)$ for all $H$.
On the other hand, $\mathcal{CO}^0_L(e_{Y,L})=e_L$ implies that $c(e_{Y,L},H) \ge \ell(e_{L},H)$ for all $H$.
By the triangle inequality, we have $\ell(e_{L}/Z,H)=\ell(e_{L}/Z,H) + \ell(vol,0) \ge \ell(vol/Z,H)$ so the result follows.
\end{proof}

\bibliographystyle{alpha}
\bibliography{biblio}

@preamble{"\def\cprime{$'$} "}

@article {Herman1,
    AUTHOR = {Herman, Michael-R.},
     TITLE = {Exemples de flots hamiltoniens dont aucune perturbation en
              topologie {$C^\infty$} n'a d'orbites p\'{e}riodiques sur un ouvert
              de surfaces d'\'{e}nergies},
   JOURNAL = {C. R. Acad. Sci. Paris S\'{e}r. I Math.},
  FJOURNAL = {Comptes Rendus de l'Acad\'{e}mie des Sciences. S\'{e}rie I.
              Math\'{e}matique},
    VOLUME = {312},
      YEAR = {1991},
    NUMBER = {13},
     PAGES = {989--994},
      ISSN = {0764-4442},
   MRCLASS = {58F05 (58F22 70H05)},
  MRNUMBER = {1113091},
MRREVIEWER = {Michel Willem},
}

@article {Herman2,
    AUTHOR = {Herman, Michael-R.},
     TITLE = {Diff\'{e}rentiabilit\'{e} optimale et contre-exemples \`a la fermeture
              en topologie {$C^\infty$} des orbites r\'{e}currentes de flots
              hamiltoniens},
   JOURNAL = {C. R. Acad. Sci. Paris S\'{e}r. I Math.},
  FJOURNAL = {Comptes Rendus de l'Acad\'{e}mie des Sciences. S\'{e}rie I.
              Math\'{e}matique},
    VOLUME = {313},
      YEAR = {1991},
    NUMBER = {1},
     PAGES = {49--51},
      ISSN = {0764-4442},
   MRCLASS = {58F05 (58F25)},
  MRNUMBER = {1115947},
}

@article {Fish-Hofer20,
    AUTHOR = {Fish, Joel W. and Hofer, Helmut H. W.},
     TITLE = {Almost existence from the feral perspective and some
              questions},
   JOURNAL = {Ergodic Theory Dynam. Systems},
  FJOURNAL = {Ergodic Theory and Dynamical Systems},
    VOLUME = {42},
      YEAR = {2022},
    NUMBER = {2},
     PAGES = {792--834},
      ISSN = {0143-3857},
   MRCLASS = {32Q65 (37C99 37J06 37J39 53D99)},
  MRNUMBER = {4362909},
       DOI = {10.1017/etds.2021.20},
       URL = {https://doi.org/10.1017/etds.2021.20},
}

@article {Entov-Polterovich09,
    AUTHOR = {Entov, Michael and Polterovich, Leonid},
     TITLE = {Rigid subsets of symplectic manifolds},
   JOURNAL = {Compos. Math.},
  FJOURNAL = {Compositio Mathematica},
    VOLUME = {145},
      YEAR = {2009},
    NUMBER = {3},
     PAGES = {773--826},
      ISSN = {0010-437X,1570-5846},
   MRCLASS = {53D40 (53D12 53D35)},
  MRNUMBER = {2507748},
MRREVIEWER = {Martin\ Pinsonnault},
       DOI = {10.1112/S0010437X0900400X},
       URL = {https://doi.org/10.1112/S0010437X0900400X},
}

@article{Xue,
  doi = {10.48550/ARXIV.2207.06208},
  
  journal = {arXiv:2207.06208},
  
  author = {Xue, Jinxin},
  
  title = {Strong closing lemma and {KAM} normal form},
  
  publisher = {arXiv},
  
  year = {2022},
  
  copyright = {Creative Commons Attribution 4.0 International}
}

@article{CDPT,
  doi = {10.48550/ARXIV.2206.04738},
  
  journal = {arXiv:2206.04738},
  
  author = {Chaidez, Julian and Datta, Ipsita and Prasad, Rohil and Tanny, Shira},
  
  title = {Contact homology and the strong closing lemma for ellipsoids},
  
  publisher = {arXiv},
  
  year = {2022},
  
  copyright = {Creative Commons Attribution 4.0 International}
}

@article {Buhovsky23,
    AUTHOR = {Buhovsky, Lev},
     TITLE = {On two remarkable groups of area-preserving homeomorphisms},
   JOURNAL = {J. Math. Phys. Anal. Geom.},
  FJOURNAL = {Journal of Mathematical Physics, Analysis, Geometry},
    VOLUME = {19},
      YEAR = {2023},
    NUMBER = {2},
     PAGES = {339--373},
      ISSN = {1812-9471,1817-5805},
   MRCLASS = {53D05},
  MRNUMBER = {4633987},
MRREVIEWER = {St\'ephane\ Tchuiaga},
}

@article {Cineli-Seyfaddini,
    AUTHOR = {{\c C}ineli, Erman and Seyfaddini, Sobhan},
     TITLE = {The strong closing lemma and {H}amiltonian pseudo-rotations},
   JOURNAL = {J. Mod. Dyn.},
  FJOURNAL = {Journal of Modern Dynamics},
    VOLUME = {20},
      YEAR = {2024},
     PAGES = {299--318},
      ISSN = {1930-5311,1930-532X},
   MRCLASS = {37J20 (37J40 53D40)},
  MRNUMBER = {4799466},
       DOI = {10.3934/jmd.2024008},
       URL = {https://doi.org/10.3934/jmd.2024008},
}

@article {Irie24,
    AUTHOR = {Irie, Kei},
     TITLE = {Strong closing property of contact forms and action selecting
              functors},
   JOURNAL = {J. Fixed Point Theory Appl.},
  FJOURNAL = {Journal of Fixed Point Theory and Applications},
    VOLUME = {26},
      YEAR = {2024},
    NUMBER = {2},
     PAGES = {Paper No. 15, 23},
      ISSN = {1661-7738,1661-7746},
   MRCLASS = {53D10 (18F99 37J55)},
  MRNUMBER = {4729894},
       DOI = {10.1007/s11784-024-01102-1},
       URL = {https://doi.org/10.1007/s11784-024-01102-1},
}

@article{EH21,
	author = {O. Edtmair and M. Hutchings},
	journal = {arXiv:2110.02463, Forum Math. Sigma (to appear)},
	title = {{PFH} spectral invariants and ${C}^{\infty}$ closing lemmas},
	year = {2021}}

@article{CPZ21,
	author = {D. Cristofaro-Gardiner and R. Prasad and B. Zhang},
	journal = {arXiv:2110.02925},
	title = {Periodic {F}loer homology and the smooth closing lemma for area-preserving surface diffeomorphisms},
	year = {2021}}

@article{MSS,
    title={Orbifold {H}amiltonian {F}loer theory for global quotients},
    author = {Cheuk Yu Mak and Sobhan Seyfaddini and Ivan Smith},
year={2025},
journal={arXiv:2502.11290},
primaryClass={math.SG},
}

@article {CHT,
    AUTHOR = {Colin, Vincent and Honda, Ko and Tian, Yin},
     TITLE = {Applications of higher-dimensional {H}eegaard {F}loer homology
              to contact topology},
   JOURNAL = {J. Topol.},
  FJOURNAL = {Journal of Topology},
    VOLUME = {17},
      YEAR = {2024},
    NUMBER = {3},
     PAGES = {Paper No. e12349, 77},
      ISSN = {1753-8416,1753-8424},
   MRCLASS = {57K43 (53D10 53D40)},
  MRNUMBER = {4775502},
       DOI = {10.1112/topo.12349},
       URL = {https://doi.org/10.1112/topo.12349},
}

@article{KY,
      title={Heegaard {F}loer {S}ymplectic homology and {V}iterbo's isomorphism theorem in the context of multiple particles}, 
      author={Roman Krutowski and Tianyu Yuan},
      year={2023},
      journal={arXiv:2311.17031},
      primaryClass={math.SG},
      url={https://arxiv.org/abs/2311.17031}, 
}

@article{Brendel-Kim,
    author = {Brendel, Jo\'e and Kim, Joontae},
    title = {Lagrangian split tori on {$S^2 \times S^2$} and billiards},
    journal = {arXiv:2502.03324},
    year = {2025}, 
}

@article {CGHMSS,
    AUTHOR = {Cristofaro-Gardiner, Daniel and Humili\`ere, Vincent and Mak,
              Cheuk Yu and Seyfaddini, Sobhan and Smith, Ivan},
     TITLE = {Quantitative {H}eegaard {F}loer cohomology and the {C}alabi
              invariant},
   JOURNAL = {Forum Math. Pi},
  FJOURNAL = {Forum of Mathematics. Pi},
    VOLUME = {10},
      YEAR = {2022},
     PAGES = {Paper No. e27, 59},
      ISSN = {2050-5086},
   MRCLASS = {53D40 (37K06)},
  MRNUMBER = {4524364},
MRREVIEWER = {Christopher\ T.\ Woodward},
       DOI = {10.1017/fmp.2022.18},
       URL = {https://doi.org/10.1017/fmp.2022.18},
}

@article{CGHMSS22,
	author = {Dan Cristofaro-Gardiner and Vincent Humili\`ere and Cheuk Yu Mak and Sobhan Seyfaddini and Ivan Smith},
	journal = {Ann. Sci. Éc. Norm. Supér. (to appear), arXiv:2206.10749},
	title = {Subleading asymptotics of link spectral invariants and homeomorphism groups of surfaces},
	year = {2022}}

@article{airie,
	author = {Asaoka, Masayuki and Irie, Kei},
	doi = {10.1007/s00039-016-0386-3},
	fjournal = {Geometric and Functional Analysis},
	issn = {1016-443X},
	journal = {Geom. Funct. Anal.},
	mrclass = {37J10 (37E30 37J45)},
	mrnumber = {3568031},
	mrreviewer = {Dimitar Angelov Kolev},
	number = {5},
	pages = {1245--1254},
	title = {A {$C^\infty$} closing lemma for {H}amiltonian diffeomorphisms of closed surfaces},
	url = {https://doi-org.ezproxy.is.ed.ac.uk/10.1007/s00039-016-0386-3},
	volume = {26},
	year = {2016}}

@article{irie,
	author = {Irie, Kei},
	doi = {10.3934/jmd.2015.9.357},
	fjournal = {Journal of Modern Dynamics},
	issn = {1930-5311},
	journal = {J. Mod. Dyn.},
	mrclass = {37J45 (53D42)},
	mrnumber = {3436746},
	mrreviewer = {Sheila Sandon},
	pages = {357--363},
	title = {Dense existence of periodic {R}eeb orbits and {ECH} spectral invariants},
	url = {https://doi.org/10.3934/jmd.2015.9.357},
	volume = {9},
	year = {2015}}

@book{Polterovich-Rosen,
	author = {Polterovich, Leonid and Rosen, Daniel},
	doi = {10.1090/crmm/034},
	isbn = {978-1-4704-1693-5},
	mrclass = {53Dxx (22E65 57R17 57R58 58E05 81P45)},
	mrnumber = {3241729},
	mrreviewer = {Charles-Michel Marle},
	pages = {xii+203},
	publisher = {American Mathematical Society, Providence, RI},
	series = {CRM Monograph Series},
	title = {Function theory on symplectic manifolds},
	url = {https://doi.org/10.1090/crmm/034},
	volume = {34},
	year = {2014}}

@article{Entov-Polterovich,
	author = {Entov, Michael and Polterovich, Leonid},
	doi = {10.1155/S1073792803210011},
	fjournal = {International Mathematics Research Notices},
	issn = {1073-7928},
	journal = {Int. Math. Res. Not.},
	mrclass = {53D35 (53D40 53D45)},
	mrnumber = {1979584},
	mrreviewer = {Ignasi Mundet-Riera},
	number = {30},
	pages = {1635--1676},
	title = {Calabi quasimorphism and quantum homology},
	url = {https://doi.org/10.1155/S1073792803210011},
	year = {2003}}

@article {Silversmith,
    AUTHOR = {Silversmith, Rob},
     TITLE = {Gromov-{W}itten invariants of {$\mathrm{Sym}^d\Bbb {P}^r$}},
   JOURNAL = {Trans. Amer. Math. Soc.},
  FJOURNAL = {Transactions of the American Mathematical Society},
    VOLUME = {376},
      YEAR = {2023},
    NUMBER = {9},
     PAGES = {6573--6622},
      ISSN = {0002-9947,1088-6850},
   MRCLASS = {14N35},
  MRNUMBER = {4630785},
MRREVIEWER = {Matthew\ B.\ Young},
       DOI = {10.1090/tran/8938},
       URL = {https://doi.org/10.1090/tran/8938},
}

@article {Pol-Shel21,
    AUTHOR = {Polterovich, Leonid and Shelukhin, Egor},
     TITLE = {Lagrangian configurations and {H}amiltonian maps},
   JOURNAL = {Compos. Math.},
  FJOURNAL = {Compositio Mathematica},
    VOLUME = {159},
      YEAR = {2023},
    NUMBER = {12},
     PAGES = {2483--2520},
      ISSN = {0010-437X,1570-5846},
   MRCLASS = {53D12 (37J06 53D05)},
  MRNUMBER = {4651502},
       DOI = {10.1112/s0010437x23007455},
       URL = {https://doi-org.sheffield.idm.oclc.org/10.1112/s0010437x23007455},
}

@article{CGHS20,
	archiveprefix = {arXiv},
	author = {Dan Cristofaro-Gardiner and Vincent Humili\`ere and Sobhan Seyfaddini},
	eprint = {2001.01792},
	journal = {Annals of Mathematics},
	primaryclass = {math.SG},
	title = {Proof of the simplicity conjecture},
	year = {2024},
    volume = {199},
    pages = {181-257},
}

@article{HLS1,
	author = {Humili\`ere, Vincent and Leclercq, R\'{e}mi and Seyfaddini, Sobhan},
	doi = {10.1215/00127094-2881701},
	fjournal = {Duke Mathematical Journal},
	issn = {0012-7094},
	journal = {Duke Math. J.},
	mrclass = {53D40 (37J05)},
	mrnumber = {3322310},
	mrreviewer = {Manuel de Le\'{o}n},
	number = {4},
	pages = {767--799},
	title = {Coisotropic rigidity and {$C^0$}-symplectic geometry},
	url = {https://doi.org/10.1215/00127094-2881701},
	volume = {164},
	year = {2015}}

@article{calabi,
	author = {Calabi, Eugenio},
	journal = {Problems in analysis ({L}ectures at the {S}ympos. in honor of {S}alomon {B}ochner, {P}rinceton {U}niv., {P}rinceton, {N}.{J}., 1969)},
	mrclass = {58D05 (53C15 57D25)},
	mrnumber = {0350776},
	mrreviewer = {Alan Weinstein},
	pages = {1--26},
	publisher = {Princeton Univ. Press, Princeton, N.J.},
	title = {On the group of automorphisms of a symplectic manifold},
	year = {1970}}

@article{Cho-Clifford,
	author = {Cho, Cheol-Hyun},
	doi = {10.1155/S1073792804132716},
	fjournal = {International Mathematics Research Notices},
	issn = {1073-7928},
	journal = {Int. Math. Res. Not.},
	mrclass = {53D40 (53D12)},
	mrnumber = {2057871},
	mrreviewer = {Michael J. Usher},
	number = {35},
	pages = {1803--1843},
	title = {Holomorphic discs, spin structures, and {F}loer cohomology of the {C}lifford torus},
	url = {https://doi-org.ezp.lib.cam.ac.uk/10.1155/S1073792804132716},
	year = {2004}}

@article{fathi,
	author = {Fathi, A.},
	fjournal = {Annales Scientifiques de l'\'Ecole Normale Sup\'erieure. Quatri\`eme S\'erie},
	issn = {0012-9593},
	journal = {Ann. Sci. \'Ecole Norm. Sup. (4)},
	mrclass = {58F11 (28D10 57S05)},
	mrnumber = {584082},
	mrreviewer = {Hans G. Bothe},
	number = {1},
	pages = {45--93},
	title = {Structure of the group of homeomorphisms preserving a good measure on a compact manifold},
	url = {http://www.numdam.org/item?id=ASENS_1980_4_13_1_45_0},
	volume = {13},
	year = {1980}}

@article {FOOOtoric,
    AUTHOR = {Fukaya, Kenji and Oh, Yong-Geun and Ohta, Hiroshi and Ono,
              Kaoru},
     TITLE = {Lagrangian {F}loer theory and mirror symmetry on compact toric
              manifolds},
   JOURNAL = {Ast\'erisque},
  FJOURNAL = {Ast\'erisque},
    NUMBER = {376},
      YEAR = {2016},
     PAGES = {vi+340},
      ISSN = {0303-1179,2492-5926},
      ISBN = {978-2-85629-825-1},
   MRCLASS = {53D40 (14M25 53D37 53D45)},
  MRNUMBER = {3460884},
MRREVIEWER = {Christopher\ T.\ Woodward},
}

@article{FOOOspectral,
	author = {Fukaya, Kenji and Oh, Yong-Geun and Ohta, Hiroshi and Ono, Kaoru},
	doi = {10.1090/memo/1254},
	fjournal = {Memoirs of the American Mathematical Society},
	isbn = {978-1-4704-3625-4; 978-1-4704-5325-1},
	issn = {0065-9266},
	journal = {Mem. Amer. Math. Soc.},
	mrclass = {53D40 (14M25 20F65 53D12 53D20 53D45)},
	mrnumber = {3986938},
	mrreviewer = {Jun Zhang},
	number = {1254},
	pages = {x+266},
	title = {Spectral invariants with bulk, quasi-morphisms and {L}agrangian {F}loer theory},
	url = {https://doi-org.ezproxy.is.ed.ac.uk/10.1090/memo/1254},
	volume = {260},
	year = {2019}}

@article {Sanda,
    AUTHOR = {Sanda, Fumihiko},
     TITLE = {Computation of quantum cohomology from {F}ukaya categories},
   JOURNAL = {Int. Math. Res. Not. IMRN},
  FJOURNAL = {International Mathematics Research Notices. IMRN},
      YEAR = {2021},
    NUMBER = {1},
     PAGES = {769--803},
      ISSN = {1073-7928,1687-0247},
   MRCLASS = {53D37 (18F99)},
  MRNUMBER = {4198511},
MRREVIEWER = {Hai-Long\ Her},
       DOI = {10.1093/imrn/rnaa089},
       URL = {https://doi.org/10.1093/imrn/rnaa089},
}

@article {SheridanFano,
    AUTHOR = {Sheridan, Nick},
     TITLE = {On the {F}ukaya category of a {F}ano hypersurface in
              projective space},
   JOURNAL = {Publ. Math. Inst. Hautes \'Etudes Sci.},
  FJOURNAL = {Publications Math\'ematiques. Institut de Hautes \'Etudes
              Scientifiques},
    VOLUME = {124},
      YEAR = {2016},
     PAGES = {165--317},
      ISSN = {0073-8301,1618-1913},
   MRCLASS = {53D37 (14F05 14J33 14N35)},
  MRNUMBER = {3578916},
MRREVIEWER = {Christian\ Lehn},
       DOI = {10.1007/s10240-016-0082-8},
       URL = {https://doi.org/10.1007/s10240-016-0082-8},
}

@article{Hirschi-Hugtenburg,
    author = {Hirschi, Amanda and Hugtenburg, Kai} ,
    title = {An open-closed {D}eligne-{M}umford field theory associated to a {L}agrangian submanifold},
    journal = {arXiv:2501.04687},
    year = {2025}
}

@article{MS19,
	author = {Mak, Cheuk Yu and Smith, Ivan},
	journal = {Geom. Funct. Anal.},
	title = {Non-displaceable {L}agrangian links in four-manifolds},
volume={31},
pages={438-481},
	year = {2021}}

@book{mcduff-salamon,
	author = {McDuff, Dusa and Salamon, Dietmar},
	doi = {10.1093/oso/9780198794899.001.0001},
	edition = {Third},
	isbn = {978-0-19-879490-5; 978-0-19-879489-9},
	mrclass = {53D35 (53D40 57R17 57R57 57R58)},
	mrnumber = {3674984},
	mrreviewer = {Hansj\"{o}rg Geiges},
	pages = {xi+623},
	publisher = {Oxford University Press, Oxford},
	series = {Oxford Graduate Texts in Mathematics},
	title = {Introduction to symplectic topology},
	url = {https://doi.org/10.1093/oso/9780198794899.001.0001},
	year = {2017}}

@article{muller-oh,
	author = {Oh, Yong-Geun and M{\"u}ller, Stefan},
	fjournal = {The Journal of Symplectic Geometry},
	issn = {1527-5256},
	journal = {J. Symplectic Geom.},
	mrclass = {53D35},
	mrnumber = {2377251 (2009k:53227)},
	mrreviewer = {Martin Pinsonnault},
	number = {2},
	pages = {167--219},
	title = {The group of {H}amiltonian homeomorphisms and {$C^0$}--symplectic topology},
	url = {http://projecteuclid.org/getRecord?id=euclid.jsg/1202004455},
	volume = {5},
	year = {2007}}

@article{usher11,
	author = {Usher, Michael},
	doi = {10.2140/gt.2011.15.1313},
	fjournal = {Geometry \& Topology},
	issn = {1465-3060},
	journal = {Geom. Topol.},
	mrclass = {53D45 (53D40)},
	mrnumber = {2825315},
	mrreviewer = {David E. Hurtubise},
	number = {3},
	pages = {1313--1417},
	title = {Deformed {H}amiltonian {F}loer theory, capacity estimates and {C}alabi quasimorphisms},
	url = {https://doi.org/10.2140/gt.2011.15.1313},
	volume = {15},
	year = {2011}}

@article {Cho-Poddar,
    AUTHOR = {Cho, Cheol-Hyun and Poddar, Mainak},
     TITLE = {Holomorphic orbi-discs and {L}agrangian {F}loer cohomology of
              symplectic toric orbifolds},
   JOURNAL = {J. Differential Geom.},
  FJOURNAL = {Journal of Differential Geometry},
    VOLUME = {98},
      YEAR = {2014},
    NUMBER = {1},
     PAGES = {21--116},
      ISSN = {0022-040X},
   MRCLASS = {53D40 (32Q65 57R18)},
  MRNUMBER = {3263515},
MRREVIEWER = {Saibal Ganguli},
       URL = {http://projecteuclid.org/euclid.jdg/1406137695},
}

{\small

\medskip
\noindent Cheuk Yu Mak\\
\noindent School of Mathematical and Physical Sciences, Hicks Building, University of Sheffield, Sheffield, S10 2TN, UK\\
{\it e-mail:} c.mak@sheffield.ac.uk

\medskip
 \noindent Sobhan Seyfaddini\\
\noindent Department of Mathematics, ETH Zurich, 8092, Zürich,
Switzerland. \\
 {\it e-mail:}  sobhan.seyfaddini@math.ethz.ch.

 \medskip
 \noindent Ivan Smith\\
\noindent Centre for Mathematical Sciences, University of Cambridge, Wilberforce Road, CB3 0WB, U.K.\\
{\it e-mail:} is200@cam.ac.uk

}

\end{document}